\documentclass[a4paper, 11pt,reqno]{amsart}
\usepackage{amsthm}
\usepackage{amsfonts, amsthm, amssymb, amsmath, stackengine, scalerel}
\usepackage{hyperref}
\usepackage{mathrsfs,array,tikz-cd}
\usepackage{eucal,fullpage,times,color,enumerate,accents}
\usepackage[all]{xy}
\usepackage{url}
\usepackage{enumitem}
\usepackage[new]{old-arrows}
\usepackage{extpfeil}
\usepackage{verbatim}  
\usetikzlibrary{graphs,decorations.pathmorphing,decorations.markings}

\input xy
\xyoption{all}

\newextarrow{\xbigtoto}{{20}{20}{20}{20}}
{\bigRelbar\bigRelbar{\bigtwoarrowsleft\rightarrow\rightarrow}}    

\newcommand{\Q}{\mathbb{Q}}
\newcommand{\Z}{\mathbb{Z}}
\newcommand{\N}{\mathbb{N}}
\newcommand{\com}{\mathsf{c}}

\newcommand{\Set}{\mathrm{Set}}

\newcommand{\ord}{\mathrm{ord}}

\newcommand{\Loc}{\mathrm{Loc}}

\newcommand{\argu}{\text{(---)}}

\newcommand{\frap}{\mathfrak{p}}
\newcommand{\Frap}{\mathfrak{P}}

\newcommand{\gav}{|\cdot|}    
\newcommand{\Mag}[1]{\mathop{\mathrm{mag}}#1}  
\newcommand{\pchav}[2]{|#1|_{\mathbb{F}_{#2}}}  

\newcommand{\spec}{\mathrm{spec}}  
\newcommand{\ISpec}{\mathrm{ISpec}}
\newcommand{\LSpec}{\mathrm{LSpec}}
\newcommand{\FSpec}{\mathrm{FSpec}}
\newcommand{\PSpec}{\spec^+(\Z)}  
\newcommand{\M}{\mathcal{M}}

\newcommand{\AffBerk}{\mathbb{A}_{\mathrm{Berk}}}
\newcommand{\Znq}{\mathbb{Z}_{\neq 0}}

\newcommand{\lav}{\overleftarrow{\mathsf{av}}}  
\newcommand{\Dav}{\mathsf{av}}  
\newcommand{\Davpos}{\mathsf{av}^{+}}  
\newcommand{\Alav}[1]{#1\text{-}\lav}  
\newcommand{\ADav}[1]{#1\text{-}\Dav}  
\newcommand{\ADavpos}[1]{#1\text{-}\Davpos}  

\newcommand{\FL}{\mathfrak{P}_{\overleftarrow{\Lambda}}}
\newcommand{\FLDed}{\mathfrak{P}_{\Lambda}}
\newcommand{\R}{\mathbb{R}}

\newcommand{\thT}{\mathbb{T}}
\newcommand{\baseS}{\mathcal{S}}

\newcommand{\cosimp}[3]{\xymatrix@1{#1 \ar@<.4ex>[r] \ar@<-.4ex>[r] & {\ }#2 \ar@<0.8ex>[r] \ar[r] \ar@<-.8ex>[r] & {\ } #3 \ar@<1.2ex>[r] \ar@<.4ex>[r] \ar@<-.4ex>[r] \ar@<-1.2ex>[r] & \cdots }}
\newsavebox{\pullback}
\sbox\pullback{%
	\begin{tikzpicture}%
	\draw (0,0) -- (1ex,0ex);%
	\draw (1ex,0ex) -- (1ex,1ex);%
	\end{tikzpicture}}

\newcommand{\equalizer}[2]{\xymatrix@1{#1 \ar@<.4ex>[r] \ar@<-0.4ex>[r] & {\ } #2}}

\newcommand{\adjunction}[4]{\xymatrix@1{#1{\ } \ar@<-0.3ex>[r]_{ {\scriptstyle #2}} & {\ } #3 \ar@<-0.3ex>[l]_{ {\scriptstyle #4}}}}

\begin{document}
	\bibliographystyle{alpha}
	\theoremstyle{plain}
	\newtheorem{theorem}{Theorem}[section]
	\newtheorem*{theorem*}{Theorem}
	\newtheorem*{condition*}{Condition}
	\newtheorem*{definition*}{Definition}
	\newtheorem*{corollary*}{Corollary}
	\newtheorem{proposition}[theorem]{Proposition}
	\newtheorem{lemma}[theorem]{Lemma}
	\newtheorem{corollary}[theorem]{Corollary}
	\newtheorem{claim}[theorem]{Claim}
	\newtheorem{conclusion}[theorem]{Conclusion}

	\theoremstyle{definition}
	\newtheorem{definition}[theorem]{Definition}
	\newtheorem{question}{Question}
	\newtheorem{remark}[theorem]{Remark}
	\newtheorem{observation}[theorem]{Observation}
	\newtheorem{discussion}[theorem]{Discussion}
	\newtheorem{guess}[theorem]{Guess}
	\newtheorem{example}[theorem]{Example}
	\newtheorem{condition}[theorem]{Condition}
	\newtheorem{warning}[theorem]{Warning}
	\newtheorem{notation}[theorem]{Notation}
	\newtheorem{construction}[theorem]{Construction}
	\newtheorem{problem}[theorem]{Problem}
	\newtheorem{fact}[theorem]{Fact}
	\newtheorem{thesis}[theorem]{Thesis}
	\newtheorem{convention}[theorem]{Convention}
	
	\newcommand{\MING}[1]{{\color{purple} {\tiny \bf (M:)} {\bf #1}}}
	
	\title{A Point-Free Look at Ostrowski's Theorem and Absolute Values}
	\author{Ming Ng and Steven Vickers}
    
    \thanks{\emph{Thanks:} Research partially supported by EPSRC Grant EP/V028812/1.} 
	
\begin{abstract}
    This paper investigates the absolute values on $\Z$ valued in the upper reals (i.e. reals for which only a right Dedekind section is given).
    These necessarily include multiplicative seminorms corresponding to the finite prime fields $\mathbb{F}_p$.
    As an Ostrowski-type Theorem, the space of such absolute values is homeomorphic to a space of prime ideals (with co-Zariski topology) suitably paired with upper reals in the range $[-\infty, 1]$,
    and from this is recovered the standard Ostrowski's Theorem for absolute values on $\Q$. 
    
    Our approach is fully constructive, using, in the topos-theoretic sense, geometric reasoning with point-free spaces, and that calls for a careful distinction between Dedekinds vs. upper reals.
    This forces attention on topological subtleties that are obscured in the classical treatment.
    In particular, the admission of multiplicative seminorms points to connections with Berkovich and adic spectra.
    The results are also intended to contribute to characterising a (point-free) space of places of $\Q$.
\end{abstract}

\maketitle

Classically, an absolute value on $\Q$ is defined as a multiplicative norm $\gav\colon \Q\rightarrow [0,\infty)$, and we have a complete classification of all such absolute values via Ostrowski's Theorem: up to equivalence, they are all either Euclidean, or $p$-adic, or trivial.
On a basic level, this paper can be read as just a piece of constructive mathematics. Our work involves sharpening Ostrowski's theorem in various sensible ways -- e.g. by rephrasing Ostrowski's Theorem as a representation result (instead of just a classification result), or by adhering to an inherently topological sense of geometric reasoning (instead of working classically), extending previous work from~\cite{NV} .

However, rather more interestingly, our investigations also bring to light a subtle connection between topology and algebra previously elided by classical assumptions, raising challenging implications for the foundations of arithmetic geometry. 

Let us elaborate. First, geometric mathematics presents us with various \emph{classically} equivalent but \emph{constructively} inequivalent notions of the reals. As such, there exists several options for reworking the notion of an absolute value geometrically -- e.g. should the map be valued in Dedekinds or one-sided reals
(see Section~\ref{sec:UpperReals})?
Further, since $\gav\colon\Q\rightarrow [0,\infty)$ is determined by its values on the integers, one may also define an absolute value as a map from either $\Z$ or $\Q$. This selection of topological and algebraic options have a curious interaction, which we summarise in the following observation.

\begin{observation}\label{obs:absvalueQ} \hfill
	\begin{enumerate}[label=(\roman*)]
		\item An absolute value on $\Q$ (or indeed any field) must be valued in Dedekinds, and not the one-sideds.
		\item An absolute value on $\Z$ can be valued in the one-sided reals.  
	\end{enumerate}
\end{observation}  

The justification for this will be seen in Proposition~\ref{prop:fieldav}, after the relevant definitions have been introduced. For now, the statement of Observation~\ref{obs:absvalueQ} sets the scope for our present task: work towards an explicit description of the space of absolute values, first by proving a modified version of Ostrowski's theorem for upper-valued absolute values on $\mathbb{Z}$
\[\gav\colon \Z\rightarrow \overleftarrow{[0,\infty)},\]
before recovering the standard Ostrowski's Theorem for  absolute values on $\Q$. 

Before proceeding, a few obligatory remarks regarding motivation. The decision to use upper reals rather than Dedekinds may strike the lay reader as constructivist hair-splitting, but in fact it ties together two \textit{a priori} unrelated mathematical threads:
\begin{enumerate}[label=(\alph*)]
	\item Vickers \cite{ViLoccompI}: To provide a geometric account of the completions of a (generalised) metric space, it suffices for the metric to be valued in non-negative upper reals (as opposed to the Dedekinds).
	\item Berkovich \cite{BerkovichMonograph}: For a suitable\footnote{By suitable, let us mean: non-Archimedean, non-trivially valued, algebraically closed and complete.} field $K$, every point $x$ of the Berkovich affine line $\AffBerk^1$ corresponds to a nested descending sequence of closed discs in $K$:
	\[D_1\supseteq D_2 \supseteq \dots\]
\end{enumerate}
This connection is made precise in \cite{MingPhD} (and in a planned forthcoming paper), which gives an example of how the structural gap between trivially vs. non-trivially valued fields in Berkovich geometry can
 (surprisingly) be eliminated via point-free techniques.%
\footnote{The reader familiar with Berkovich geometry should notice the suggestive parallel between ``a point of $\AffBerk^1$ = a nested sequence of discs'' and ``a point of the upper reals = a rounded ideal'' (see e.g. \cite[Remark 1.30]{NV}). 
 } 

There is a more fundamental question: why even rework these algebraic ideas geometrically in the first place? Our answer lies in our wish to exploit the tight connections between geometric logic and toposes. A key structure theorem in topos theory tells us that the models of any geometric theory are classified by some topos, and that any topos classifies the models of some geometric theory.%
\footnote{
  Technically: ``topos'' here means a Grothendieck topos,
  i.e. a bounded $\baseS$-topos where the fixed base $\baseS$ can be any elementary topos with natural numbers object.
}
In other words, if absolute values can be described as models of some geometric theory, then we have at our disposal a deep collection of topos-theoretic tools that allows us to extract topological information from our logical setup.
This insight has been developed in~\cite{MingPhD}, which utilised descent techniques to reveal some striking differences between the Archimedean vs. non-Archimedean places of $\Q$.
Again, the upper reals play a key role here, and their presence raises challenging questions regarding certain longstanding assumptions in number theory.


\tableofcontents	

\section{Preliminaries} \label{sec:Prelim}

Methodologically, the present paper is an exercise in point-free analysis, developed using techniques of geometric logic.

Already in~\cite{NV} we have described those techniques and used them to construct real exponentiation and logarithms, so for a detailed introduction we shall refer the reader to that paper and to other references mentioned there.
For the present we shall content ourselves with noting the more prominent oddities that might otherwise distract the reader.

\subsection{Point-free topological spaces} \label{sec:PtfreeSpaces}
Geometric mathematics is a particular regime of constructive mathematics tied to the (logical) perspective that a topos can be regarded as a ``point-free space'' in the following sense: 

\begin{definition}\label{def:pointfree space}
  A \emph{(point-free) space} is described by a geometric theory $\thT$, whose models are the points of the space.
  We write $[\thT]$ for the space of models of $\thT$.
    
  A map $f\colon X\rightarrow Y$ is defined by a geometric construction of points $f(x)\in Y$ out of points $x\in X$. 
\end{definition}

One way to read Definition~\ref{def:pointfree space} is that it pulls the usual notion of a \emph{space} away from its underlying set theory. Points of a space are now defined as \emph{models} of a theory as opposed to \emph{elements} of a set (equipped with some chosen topology); continuous maps are now defined as \emph{geometric transformations} of such models as opposed to \emph{functions} on sets that respect the chosen topology.

Note that the points, the models, may be sought \emph{in any Grothendieck topos.}
Hence $[\thT]$ is not in a literal sense the collection of all models.
Rather, it is a ``(pseudo-)functor of models'', by analogy with the ``functor of points'' in algebraic geometry, used to replace a set of points in situations where there are insufficient global points.
(Compare this with the definition of theory in~\cite[B4.2]{J1}.)
Technically, it is represented as the classifying topos $\baseS[\thT]$, the category of sheaves, but we make the notational distinction so as to differentiate between maps (geometric morphisms) $[\thT]\to [\thT']$
and functors $\baseS[\thT]\to \baseS[\thT']$.

\begin{definition}[Geometric Theories]
\label{def:geomlogic}
  There are various ways to define what a geometric theory is; compare, for example, \cite[D1.1.6]{J2} and~\cite[B4.2.7]{J1}.
  We set out here some key features. For further details of our usage, see~\cite{NV} or~\cite{VickersPtfreePtwise}.
  
  At its most basic, a geometric theory is a many sorted, first-order theory.
  However, there are some special features and restrictions in how its extralogical axioms may be expressed.
  
  First, there is a two-level distinction between \emph{formulae} and \emph{sequents}.
    \begin{itemize}
        \item
        A \emph{geometric formula} is a logical formula built up from the symbols in the signature and a context of finitely many free variables,
        using truth $\top$, equality $=$, finite conjunctions $\land$, arbitrary (possibly infinite) disjunctions $\bigvee$, and $\exists$.
        \item
        A \emph{geometric sequent} is an expression of the form
        $\forall xyz\ldots(\phi\rightarrow \psi)$,
        where $\phi$ and $\psi$ are geometric formulae in the same context $\{x,y,z,\ldots\}$.
    \end{itemize}
        \item
  The theory's axioms are in the form of sequents.
  Thus negation, implication, and universal quantification are expressed in the theory at a single level -- they cannot be nested.
  
  Adding extra axioms adds extra constraints on the models, and this provides the point-free notion of \emph{subspace.} Of course, we cannot simply talk about a sub\emph{set} of the points.
  
  Second, disjunctions (but not conjunctions) are allowed to be infinite.
  This provides the link with topology,
  with the propositional formulae (those without free variables) corresponding to the opens.
  The signature of a theory supplies subbasic opens, and for the rest the finite $\land$ and the arbitrary $\bigvee$ correspond to intersections and unions of opens.
  
  Finally, the \emph{geometric constructions}, by which maps are defined, are allowed in a theory as sort constructors.  
  Colimits, finite limits (note the analogy with disjunctions and finite conjunctions) and free algebra constructions are all geometric, and geometric constructions include the natural numbers $\mathbb{N}$, the integers $\mathbb{Z}$ and the rationals $\mathbb{Q}$, along with their usual arithmetic structure (e.g. addition, multiplication, strict order etc.).
\end{definition}

The basic deal is that if you reason geometrically, then all your constructions are continuous.
You also find the very notion of topological space is greatly generalised, to that of topos,
and proper classes such as those of sets or rings can be described as generalised spaces.
For the purposes of the present paper, however, we shall restrict ourselves to the category $\Loc$ of \emph{localic} spaces, i.e. spaces of models of a theory equivalent a propositional one. (For more background, see e.g. \cite{Vi4} or \cite[Ch. 2]{MingPhD}).

On the other hand, the price of that deal is a greatly reduced ability to construct sets.
For instance, powersets, the real line $\R$, and sets of functions, are all non-geometric constructions, i.e. they are not geometrically definable \emph{as sets.} These constructions still enter into our geometric reasoning, but they must be described as spaces (via geometric theories), and have intrinsic non-discrete topologies that cannot be stripped away.
A phrase ``the space of ...'' will presuppose the existence of some geometric theory whose models are the ``...''s.

\begin{discussion}[Decidability and Topology]\label{dis:geomCONSTRUCTIVE} Let $\phi$ be a  propositional formula (no free variables) definable in $\thT$, and $M$ be a model of $\thT$. Classically, the Law of Excluded Middle $\phi\vee\lnot\phi$ automatically gives that $\phi$ evaluates in $M$ as either $\top$ or $\bot$. However, in geometric logic, we are unable to even express this principle as stated (much less affirm its validity) since our syntax lacks negation. We must therefore look for alternative formulations.
    
In our setting, the validity of non-constructive principles (such as LEM) gets reframed as a question of topology. To illustrate, suppose $U$ is a propositional formula, written ``$U$'' because, in point-free topology, we regard ``the collection of models in $[\thT]$ satisfying $U$'' as an open subspace of $[\thT]$. Hence, the question of whether the logical axiom
    \[\top\to U\vee\lnot U\]
holds in $\thT$ should properly be understood as asking if there is an open $V$ that is a Boolean complement of $V$ in $[\thT]$
($U\wedge V \rightarrow \bot$, $\top \rightarrow U\vee V$).
We call $U$ \emph{decidable} in this case. As we shall later see, in situations where a desired property is not decidable, this often creates subtleties when trying to prove analogues of classical results in our setting.
 On the other hand, if the relevant opens are decidable, then the definitions can be made by case splits, $U$ versus $V$.
 (Categorically, $[\thT]$ is then a coproduct of $U$ and $V$.)
 For example, equality on the rationals is decidable. 

More generally, we must be more careful, but it is often helpful to look at sub\emph{spaces} of 1. There is a lattice of subspaces
(for an account from the geometric point of view see~\cite{ViSublocales}),
and there every open $U$ of a space $X$ (with sequent $\top \rightarrow U$) does have a \emph{closed complement} $U^\com$, with sequent $U\rightarrow\bot$. This is a Boolean complement in the lattice of subspaces. In other words, decidability of $U$ means it is \emph{clopen}, both open and closed (because it is $V^\com$).
\end{discussion}

Although we cannot define maps from $X$ by cases based on $U$ and $U^\com$, the following lemma shows we can use cases to prove properties corresponding to subspaces. That's because if an open $U$ and its closed complement $U^\com$ are both contained in some subspace, then so must be their join in the subspace lattice, and that is the whole space $X$. We emphasise that this does not allow us to say that every point in $X$ is either in $U$ or in $U^\com$ -- that may fail to be (constructively) true.

\begin{lemma}[Case-Splitting Lemma]\label{lem:casesplitting} Consider the following cospan in $\Loc$:
    \begin{equation}\label{eq:cospan}
        \begin{tikzcd}
            & Y \ar[d,hook,"i"]\\
            X \ar[r,swap,"f"] & Z
        \end{tikzcd}
    \end{equation}
    where $i$ is a subspace embedding. Further, suppose that:
    \begin{itemize}
        \item  $X=U\vee U^\com$, where $U$ is an open subspace of $X$ and $U^\com$ is its closed complement,
        with subspace embeddings
        $i_1\colon U\hookrightarrow X$ and $i_2\colon U^\com\hookrightarrow X$.
        \item
          There are maps
          $f_1\colon U\rightarrow Y$ and $f_2\colon U^\com\rightarrow Y$,
          with $f\circ i_1 = i\circ f_1$ and $f\circ i_2 = i\circ f_2$. 
    \end{itemize}
    Then the pullback $P$ of the cospan in Equation~\eqref{eq:cospan} is isomorphic to $X$.
\end{lemma}
\begin{proof}    
  We have pullbacks of spaces, so the pullback $P$ of Diagram~\eqref{eq:cospan} exists:
    \begin{equation}\label{eq:pullbackCASE}
        \begin{tikzcd}
            P \arrow{r}{\widehat{f}} \arrow[hook]{d}[swap]{p} \arrow[dr, phantom, "\usebox\pullback", very near start, color=black] & Y \arrow[hook]{d}{i}\\
            X \arrow{r}[swap]{f}
            & Z
        \end{tikzcd}
    \end{equation}
    Further (see e.g. \cite[\S II.2.3]{StoneSp}), embeddings of (localic) spaces are precisely the regular monics in $\Loc$. Since regular monics are preserved by pullback, this implies the map $p\colon P\rightarrow X$ of Diagram~\eqref{eq:pullbackCASE} is a regular monic as well. In English, this means: the pullback $P$ is a subspace of $X$.
    
    Exploiting the universal pullback property, we obtain the following diagrams:
    
    \begin{equation}
        \begin{tikzcd}
            U
            \arrow[drr, red, bend left, "f_1"]
            \arrow[ddr, red, hook, bend right, swap, "i_1"]
            \arrow[dr, dotted, "\theta_1" description] & & \\
            &P \arrow{r}{\widehat{f}} \arrow[hook]{d}[swap]{p} \arrow[dr, phantom, "\usebox\pullback", very near start, color=black] & Y \arrow[hook]{d}{i}\\
            &X\arrow{r}[swap]{f}
            & Z
        \end{tikzcd}
        \begin{tikzcd}
            U^\com
            \arrow[drr, blue, bend left, "f_2"]
            \arrow[ddr, blue, hook, bend right, swap, "i_2"]
            \arrow[dr, dotted, "\theta_2" description] & & \\
            &P \arrow{r}{\widehat{f}} \arrow[hook]{d}[swap]{p} \arrow[dr, phantom, "\usebox\pullback", very near start, color=black] & Y \arrow[hook]{d}{i}\\
            & X \arrow{r}[swap]{f}
            & Z
        \end{tikzcd}
    \end{equation}
    Since $i_1=p\circ \theta_1$ and $i_2=p\circ \theta_2$ are regular monics, this implies $\theta_1$ and $\theta_2$ are regular monics as well, i.e. $U$ and $U^\com$ are subspaces of $P$.
    Hence in the lattice of subspaces of $X$, we have $X = U \vee U^\com \leq P$, so $p$ is an isomorphism.
\end{proof}

\subsection{One-sided reals} \label{sec:UpperReals}

As is standard in point-free topology, we shall take the real numbers to be Dedekind sections of the rationals.
The two halves $(L,U)$ of a section are taken to be \emph{rounded}, for instance if $q\in L$ then there is some $q'>q$ with $q'\in L$.
This has the effect that the definition of real numbers as Dedekind sections also defines the Euclidean topology: each rational $q$ defines the opens $(q,\infty)$ as those sections for which $q\in L$, and $(-\infty, q)$ for $q\in U$.
Of course, this is typical of point-free spaces: the points are the models of some geometric theory, and then the geometric formulae are the opens.

Having two halves $L$ and $U$ corresponds to approximating reals from below and above, and classically, each can be derived from the other.
Geometrically, however, we also have ``one-sided'' approximation, lower ($L$) or upper ($U$), and they are not (constructively) equivalent to Dedekind sections.
Qualitatively they correspond to the topologies of lower or upper semicontinuity.


First, a series of remarks on order and topology in the upper reals.

\begin{convention} In any space we have a \emph{specialisation order}, $x \sqsubseteq y$ if every open containing $x$ also contains $y$.
	For upper reals this comes down to
	$\forall q\in\Q\, (x < q \rightarrow y < q)$.
	This says that, numerically, $x\geq y$, so we shall write $\geq$ and $\leq$ to denote the specialisation order and its opposite.
	For lower reals it is the other way round: specialisation order is $\leq$.
	
	In recognition of this, we use notation with arrows to show the upward direction of the specialisation order,
	for example, $\overleftarrow{[-\infty,\infty]}$ for the upper reals from $-\infty$ to $\infty$ (inclusive).
	Thus the pairs $(x,y)$ such that $x\leq y$ form a subspace of
	$\overleftarrow{[-\infty,\infty]} \times \overleftarrow{[-\infty,\infty]}$.
	
\end{convention}

\begin{remark}[The Existence of Infima/Suprema] \label{rem:infsupDed}\hfill 
	
	\begin{enumerate}[label=(\roman*)]
		\item  An important fact about point-free spaces is that they have all sups of families of points directed with respect to the specialisation order.%
		\footnote{
			In point-set topology this is an aspect of what is known as \emph{sobriety.}
		}
		For upper reals, these are directed \emph{infima.}
		In fact, since the rationals are totally ordered, this implies the existence of arbitrary infima -- at least if $\infty$ is included, as the nullary $\inf$.
		\item Classically, it is taken as characteristic of reals that they have all bounded infima and suprema.
		Constructively this is not the case.
		By item (i) we have $\inf$s of upper reals and $\sup$s of lower reals.
		For Dedekind reals, we have binary $\min$ and $\max$, but for an infinitary $\inf$ or $\sup$ of Dedekinds we get, in general, only an upper or lower real respectively.
	
	\end{enumerate}
  
\end{remark}

\begin{remark} It is clear from Remark~\ref{rem:infsupDed} that 
	we cannot capture strict inequality $x<y$ for upper reals.
	Fixing $x$, there cannot be a space $\overleftarrow{(x,\infty]}$ because it would have to contain all the rationals bigger than $x$, and hence also their infimum, which is $x$ itself.
	However, we do have a special case for when $y$ is rational, as, for each $q$, $x<q$ defines an open $\overleftarrow{[-\infty, q)}$ of $\overleftarrow{[-\infty,\infty]}$: it corresponds to the geometric sequent $\top \rightarrow x<q$.
	Its closed complement, which we shall write $\overleftarrow{[q, \infty)}$, corresponds to the sequent (not a geometric formula)
	$x<q \rightarrow \bot$.
	It is easy to show that it is equivalent to $q\leq x$, i.e. $\forall r\in\Q(x<r \rightarrow q<r)$.
	
\end{remark}

The one-sided reals have rich arithmetic structure, developed in~\cite{NV} as far as exponentiation and logarithms. However, there is one big proviso.
All our maps are continuous, and continuous maps preserve specialisation order. Hence all the arithmetic operations, to be definable on one-sided reals, have to be monotone (or antitone if they map from upper to lower or vice versa).
For example, upper reals can be added, but not subtracted (which would be antitone in the second argument).
Non-negative upper reals can be multiplied, but this cannot be extended to negatives; and even non-negative upper reals cannot be divided.

For exponentiation and logarithms these conditions can get quite intricate.
The main result we shall take from~\cite{NV} is the following.
Note the hypothesis $1\leq x$. For $0<x<1$, the exponential $x^\lambda$ would be antitone in $\lambda$.

\begin{theorem}[Upper Real Exponents]\label{thm:exponentiation} Fix a Dedekind $x\in[1,\infty)$. There exists
    the following exponentiation map on the upper reals
    \begin{align*}
        x^{\argu}\colon \overleftarrow{[-\infty,\infty)}&\longrightarrow \overleftarrow{[0,\infty)}
    \end{align*}
    satisfying the Basic Equations:
    \begin{align*}
        x^{\lambda+\lambda'}=x^\lambda x^{\lambda'}\text{, }
        \quad &x^0=1
        \\
        x^{\lambda\cdot \lambda'}=(x^\lambda)^{\lambda'}\text{, }
        \quad 
        &x^1 = x\\
        (xy)^\lambda = x^\lambda y^\lambda\text{, }
        \quad &1^\lambda = 1.
    \end{align*}
    Further, the map $x^\argu$ is (weakly) monotonic, i.e. $\lambda\leq\lambda'$ implies $x^\lambda\leq x^{\lambda'}$.
\end{theorem}
\begin{proof} The fact that $x^\argu$ is well-defined and satisfies the listed basic equations is \cite[Prop. 3.4]{NV}. The fact that $x^\argu$ is monotonic come for free since $x^\argu$, being a continuous map, must respect the specialisation order of $\overleftarrow{[-\infty,\infty)}$.
\end{proof}
The exponential map on upper reals also has an inverse $\log_x$.

Finally, exponentiation and logarithms on the usual Dedekind reals behave as expected in the geometric setting. This was worked out explicitly in Sections 3 and 4 in \cite{NV}, and will become relevant in 
Section~\ref{sec:Qav}.

\subsection{Spectra}\label{subsec:localicprimes}

Classically, Ostrowski's Theorem holds that all non-Archimedean absolute values are equivalent to the $p$-adic absolute values $\gav_p$, for some prime $p$. As such, we need a way of talking about the primes in our setting.
Classically, one considers the set of prime ideals, perhaps equipped with an additional topology (e.g. Zariski, coZariski, constructible).
In the point-free setting, however, the topology and points must be defined simultaneously and cannot be thus separated.
The following three examples develop this remark,
following the notation and topos-theoretic methods of
Cole~\cite{ColeSpectra} and Johnstone~\cite[\S 6.5]{J0}.
In all three examples, $R$ denotes a commutative ring with 1
, and $\spec(R)$ denotes the classical set of prime ideals of $R$. 

\begin{example}[The Zariski Spectrum]\label{ex:LSpec} The \emph{Zariski Spectrum} $\LSpec(R)$ for $R$ is the space whose points are the \emph{prime filters} of $R$. More explicitly, they are the subsets $S$ of $R$ satisfying the axioms:
    \begin{itemize}
        \item $1\in S$, \quad and \quad $0\notin S$;
        \item $(\forall a,a'\in R).\,\, aa'\in S \leftrightarrow a\in S\land a'\in S$;
        \item $(\forall a,a'\in R).\,\, a+a'\in S \rightarrow a\in S\vee a'\in S$.
    \end{itemize}
    Notice: in $\Set$, these axioms say that $S$ is the complement of a prime ideal of $R$.%
    \footnote{
      For the avoidance of doubt, we do not accept $R$ as a prime ideal of itself.
    }
    Hence, working classically, we get a point-set topological space equivalent to $\spec(R)$ equipped with the \emph{Zariski topology}, generated by the basic Zariski open sets $D(a)=\{\frap \in \spec (R) \,|\, a\notin \frap \}$.
\end{example}

\begin{example}[The coZariski Spectrum]\label{ex:ISpec} The \emph{coZariski Spectrum} $\ISpec(R)$ for $R$ is the space whose points are the \emph{prime ideals} of $R$.
The axioms can be obtained using contrapositives of those for $\LSpec(R)$.    
Regarded as a point-set space, $\ISpec(R)$ is $\spec(R)$ equipped with the \emph{coZariski topology}, which is generated by the sub-basic open set $V(a)=\{\frap\in\spec(R)\,|\, a\in \frap\}$.%
\end{example}

\begin{example}[The Constructible Spectrum]\label{ex:FSpec} The \emph{Constructible Spectrum} $\FSpec(R)$ for $R$ is the space whose points are the \emph{complemented prime ideals} of $R$. More explicitly, the points of $\FSpec(R)$ are pairs $(P,S)$ where $P$ is a prime ideal, $S$ is a prime filter and $P$ and $S$ are complements of each other (as subobjects of $R$). Regarded as a point-set space, $\FSpec(R)$ is $\spec(R)$ equipped with the \emph{constructible topology}, which is the join of the Zariski and coZariski topologies.
\end{example}

\begin{discussion} Examples~\ref{ex:LSpec} - \ref{ex:FSpec} combine to give the following picture. Viewed as \emph{point-free spaces}, the points of the three spectral spaces are (constructively) different. However, when viewed classically as \emph{point-set spaces}, then the topology and points begin to separate. On the one hand, their underlying set of points become classically equivalent: since every subset in $\Set$ is complementable, there is no longer a meaningful difference between, e.g. a complementable prime ideal vs. a prime ideal in $\Set$. On the other hand, their respective topologies (Zariski, coZariski, constructible) are still different, although the algebraic reasons for this difference are now obscured.
\end{discussion}

When proving Ostrowski's Theorem, it is most natural to regard the $p$-adic absolute values $\gav_p$ as corresponding to the prime ideals of $\Z$. In light of our above discussion, this indicates that the correct space for us to use is the coZariski spectrum $\ISpec(\Z)$.
Note that this is different from the standard  choice of the Zariski spectrum in classical algebraic geometry,
where we are interested in localising by inverting all the elements of a prime \emph{filter}.

\begin{definition} \label{def:PSpec}
    We write $\PSpec$ for the set of non-zero prime elements of $\Z$.
    It is geometrically isomorphic to $\N$, and we have an obvious map from $\PSpec$ to $\FSpec(\Z)$ and hence also to $\ISpec(\Z)$ and to $\LSpec(\Z)$.
\end{definition}

\emph{Classically,} the three spectra for $\Z$ can be understood as three topologies on $\PSpec\cup\{0\}$.
\begin{itemize}
    \item
    For $\ISpec(\Z)$, the opens are arbitrary subsets of $\PSpec$, together with the whole space. $0$ is bottom in the specialisation order.
    \item
    For $\LSpec(\Z)$, the opens are the empty set, and the sets containing $0$ and all but finitely many of the non-zero primes. $0$ is top in the specialisation order.
    \item
    $\FSpec(\Z)$ is the one-point compactification of $\PSpec$, with topology generated by the other two. It is $T_1$ (specialisation order is discrete).
\end{itemize}

For $\ISpec$, we can extract part of that picture geometrically.

\begin{lemma}\label{lem:primesemiringideal}
  Let $\ISpec^+(\Z)$ be the (open) subspace of $\ISpec^+(\Z)$ comprising those prime ideals that contain a non-zero element.
  Then $\ISpec^+(\Z)\cong \PSpec$.
\end{lemma}
\begin{proof}
    Let $\frap\in\ISpec(\Z)$ be a prime ideal containing a non-zero element $a$.
    We show that $\frap=(p)$ for some unique prime number $p\in \PSpec$.
    Since also $-a\in\frap$, we can assume without loss of generality that $a$ is positive. 
    By unique prime factorisation, we can represent $a$ as
    \[
      a = p_1^{\alpha_1}p_2^{\alpha_2}
                \dots p_n^{\alpha_n}\text{.}
    \]
    Now by primeness of $\frap$, we know there is some $p\in\PSpec$ dividing $a$ with $(p)\subseteq\frap$.
    
    To show that $\frap\subseteq (p)$, suppose $b\in\frap$.
    Recall that  B\'{e}zout's Identity can be (constructively) obtained from inverting the Euclidean Algorithm and performing the relevant substitutions. One can thus verify that there exist $m,n\in\Z$ such that
    \begin{equation}\label{eq:BezoutId}
        \gcd(b,p) = mb+ np \in \frap \text{.}
    \end{equation}
    Because $p$ is prime, that gcd must be either 1 or $p$,
    but 1 is not in $\frap$.   Hence $\gcd(b,p)=p$ and $b\in(p)$.
\end{proof}

\begin{remark} Classically, one typically proves that $\frap=(p)$ for some prime $p\in\mathbb{Z}$ by obtaining it as an easy corollary of the more general result that all ideals of $\mathbb{Z}$ are principal. However, proving the latter typically invokes the assumption that we can pick the least element of any non-trivial ideal $I\subset\mathbb{Z}$ (see, e.g. \cite[\S 3.7]{vdW1}), which is a non-geometric assumption since membership of $I$ is not decidable. 
\end{remark}	

\begin{remark}\label{rem:NATNUMprimeideals}
  Since a prime ideal of $\Z$ is determined by its positive elements, we might also consider a notion of prime ideal of $\N_+$. In fact, an analogue of Lemma~\ref{lem:primesemiringideal} also exists for prime ideals of the positive integers $\N_+$, with now:
    \begin{itemize}
        \item  Instead of \textit{non-trivial} prime ideals of $\Z$ we consider \textit{inhabited} prime ideals of $\N_+$;
        \item  We shall need to explicitly require that these prime ideals $\frap\in\ISpec(\N_+)$ are also closed under formal subtraction, that is:
        \[
          \forall m, n\in \N_+\text{ . }
            m\in\frap \wedge m+n\in\frap
              \rightarrow n\in \frap
          \text{.}
        \]
    \end{itemize}
    Then, the same argument works, so long as we avoid using the negative coefficients.
    The Euclidean algorithm, for finding the $\gcd$ of two natural numbers $a_0$ and $a_1$, relies on repeated integer division to produce a sequence of natural numbers with $a_{i-1} = ma_i + a_{i+1}$ and
    $0 \leq a_{i+1} < a_i$. This eventually hits $0$ at some $a_{n+1}$, with $a_n$ dividing $a_{n-1}$,
    and then $a_n = \gcd(a_0,a_1)$.
    The formal subtraction rule ensures that if $\frap$ contains $a_0$ and $a_1$ then it contains all the $a_i$s up to $a_n$.
    If $a_1$ is the prime $p$, then $a_n\in\frap$ must be either 1 -- which is impossible -- or $p$, so $p$ divides $a_0$.
\end{remark}

\section{Geometric Theories of Absolute Values}
\label{sec:GeomThAV}
Our aim is to examine the notion of absolute value when using upper reals instead of Dedekinds, and our basic general notion will be as follows.

\begin{definition} [Absolute Values, valued in upper reals]
\label{def:toposabsvUPPER}
  Let $R$ be a commutative ring (discrete, not topological).
  An \emph{absolute value} on $R$ is a map
  $\gav\colon R \to \overleftarrow{[0,\infty)}$ satisfying the following conditions:
  \[
    \begin{array}{ll}
      |0| = 0 \\
      |1| = 1 \\
      |xy| = |x||y| & \text{(Multiplicativity)} \\
      |x+y| \leq |x|+|y| & \text{(Triangle Inequality)}
    \end{array}
  \]
  We write $\Alav{R}$ for the geometric theory of absolute values on $R$, $[\Alav{R}]$ for the corresponding point-free space.  We also write $\lav$ for $\Alav{\Z}$. 
\end{definition}

We have expressed this in a form that brings out the mathematics of absolute values rather than the geometric logic.
Nonetheless, there is a geometric theory implicit in the definition. The signature has two sorts, for $R$ and for the positive rationals $\Q_+$, and a predicate $\phi$ on $R\times\Q_+$ with $\phi(x,q)$ expressing $|x|<q$.
Then the conditions listed, together with the requirement for $|x|$ to be an upper real, can all be expressed geometrically.

From the conditions, we can also derive the following.
\begin{itemize}
\item
  $|-1| = 1$. 
  
  \noindent [Why? We have $|-1|^2 = |(-1)^2| = 1$, and the only solution in the non-negative upper reals is 1.]
\item
  $|-x| = |-1||x| = |x|$.
\end{itemize}

\begin{definition} \label{def:avDed}
  A \emph{Dedekind} absolute value $\gav$ on $R$ is one that factors via $[0,\infty)$.
  We write $\ADav{R}$ for the theory of Dedekind absolute values on $R$.
  
  A Dedekind absolute value is \emph{positive definite} if it satisfies the axiom
  \[
    \forall x\in R \,(\top \rightarrow x=0 \vee |x|>0)
    \text{.} 
  \]
  We write $\ADavpos{R}$ for the theory of positive definite Dedekind absolute values on $R$.
\end{definition}

\begin{discussion}[Embeddings vs. Monics]\label{dis:Dedtouppermonic}
    Geometrically, we need to distinguish between properties defined by geometric sequents and those that also rely on uniquely defined structure.
    The former define subspaces, and corresponding embeddings, while the latter define maps of a more general kind: monics, because of uniqueness of the extra structure, but not embeddings, because the extra structure changes the topology by introducing more opens.
    
    Thus $[\ADavpos{R}]$ is a subspace of $[\ADav{R}]$. By contrast, to say an upper real is Dedekind relies on \emph{structure,} the lower part.
    The map $[\ADav{R}] \to [\Alav{R}]$ that forgets the lower parts is monic, because two Dedekind reals are equal iff their right sections are equal, but not an embedding, because $[\ADav{R}]$ has a stronger topology, with extra opens provided by those lower parts.
\end{discussion}

\begin{observation}\label{obs:seminormsUPPER} The choice of axioms in $\Alav{R}$ reflects the topological constraints the upper reals puts on the algebra. In particular, notice:
	\begin{enumerate}[label=(\roman*)]
		\item
		Without the Dedekind property, $|x|>0$ is not available as a geometric formula, and our description of positive definiteness is not geometric.
		We might try instead to formulate it as
		\[
		\forall x\in R((\forall q\in\Q_+ \, |x|<q) \rightarrow x=0) \text{,}
		\]
		but the $\forall q$ there, nested inside a sequent, is not a geometric formula.
		Hence our Definition~\ref{def:toposabsvUPPER} does not embody positive definiteness.

		Put otherwise, the topology of the upper reals forces us to work with the multiplicative \emph{seminorms} for $R$, as opposed to the usual norms. 
		\item In the classical setting, the property $|1|=1$ can be derived from the others using positive definiteness.
		For multiplicativity gives us $|1|\cdot |1|=|1|$, which implies $|1|\cdot (|1|-1)=0$. Positive definiteness tells us that $|1|\neq 0$, and thus $|1|=1$. 
	\end{enumerate}
\end{observation}

The standard examples of absolute values can all be reworked to obey Definition~\ref{def:toposabsvUPPER}, which we sketch below. The syntactic details have been suppressed for readability, but one easily checks that the definitions only use arithmetic operations that are well-defined on the upper reals and satisfy the required axioms.  

\begin{example}\label{ex:standardexamples}
	We give some standard examples on $\Z$.
	Since equality there is decidable, it suffices to define $\gav$ for $n\neq 0$ since we already require that $|0|=0$ by definition.
	\begin{enumerate}[label=(\roman*)]
		\item
		The \emph{trivial absolute value} on $\Z$, denoted $\gav_0$, is defined as 
		\[|n|_0=1, \qquad\qquad\text{for all}\,\, n\neq 0.\]
		\item
		The \emph{Euclidean absolute value} on $\Z$, denoted $\gav_\infty$, is defined as the usual norm
		\[|n|_{\infty}=n, \qquad\qquad\text{for all}\,\, n\in \N_+,\]
		which is extended to the negative integers by the fact $|n|=|-n|$ for all $n\in\Z$.
		\item
		Fix some prime $p\in\mathbb{N}_+$.	By unique prime factorisation, any non-zero integer $n\in\Z_{\neq 0}$ can be represented as $n=p^rz$, where $r\in\N$, $z\in\Z$ and $\gcd(p,z)=1$. As such, define the \textit{$p$-adic ordinal}
		\[\ord_{p}(n):=\max\{r\in\mathbb{N} \,\big| \, p^r \,\text{divides} \,\, n\}.\]
		The canonical \emph{$p$-adic absolute value} on $\Z$ is then defined as
		\[|n|_p =p^{-\ord_p(n)} \,\,\,\, \qquad\text{for all} \,\, n\neq  0.\]
		\item
		Our final example here is admitted by our use of seminorms.
		Let $p\in\N_+$ be prime.
		The \emph{$p$-characteristic absolute value} on $\Z$ is then defined as
		\[
		  \pchav{n}{p} =
		  \begin{cases}
            0 & \text{if $p$ divides $n$,} \\
            1 & \text{if $p$ does not divide $n$.}
		  \end{cases}
		\]
	\end{enumerate}
\end{example}

Further examples can be obtained by exponentiating -- in fact, the classical Ostrowski's Theorem says these are the only examples.
For our setting the details are explained more precisely in Theorem~\ref{thm:ostrowskiN}.
For the moment, let us say roughly that $\gav_\infty^\lambda$ is an absolute value if $0\leq\lambda\leq 1$, with $\gav_\infty^0 = \gav_0$. Also, $\left( p^{\ord_p(n)}\right)^\lambda$ gives an absolute value for $-\infty \leq \lambda \leq 0$, with $\gav_p$ for $\lambda=-1$ and $\pchav{\cdot}{p}$ for $\lambda=-\infty$.

\begin{remark}\label{rem:toposabsvNATNUM} Why not go one step further and work with the upper-valued absolute values on $\N$? After all, as Proposition~\ref{prop:avN} will show, upper-valued absolute values on $\Z$ themselves are determined by their values on the positive integers $\N_+$. A fuller answer will be presented in Discussion~\ref{dis:FormalSubissues}.
For now, let us say that proving Ostrowski's Theorem without additive inverses results in certain technical difficulties which incline us to stick with $\Z$. 

\begin{proposition} \label{prop:avN}
  Absolute values over $\Z$ can equivalently be defined as
  maps $\gav\colon\N \to \overleftarrow{[0,\infty)}$
  satisfying $|1|=1$, multiplicativity, and the Triangle Inequality (quantified over $\N$, of course),
  together with the \emph{Subtractive Triangle Inequality}
  \[
    |m| \leq |m+n| + |n| \quad\text{ for all $n\in\N$}
    \text{.}
  \]
\end{proposition}
\begin{proof}
  If an absolute value on $\Z$ is given, then it must satisfy the subtractive inequality when restricted to $\N$.
  (In fact, the usual and subtractive inequalities are \emph{equivalent} when applied to the whole of $\Z$.)
  
  Now suppose we are given $\gav$ on $\N$ satisfying the subtractive law.
  Clearly $\gav$ must be extended to $\Z$ by $|0|=0$ and, for $n>0$, $|-n|=|n|$.
  The notation will be less confusing if we temporarily write $\Mag {n}$ for the \emph{magnitude} of $n$, i.e. $|n|_{\infty}$ from Examples~\ref{ex:standardexamples}.
  Then $|n| = |\Mag {n}|$.
  
  The only non-obvious question is whether it satisfies the Triangle Inequality on the whole of $\Z$.
  Suppose, then, we are given $m,n\in\Z$.
  
  We claim that, of the three magnitudes of $m,n,m+n$, one is the sum of the other two.
  Note that if it is true for $m$ and $n$, then it is also true for $-m$ and $-n$: because the three magnitudes are unchanged.
  
  Let us consider how many of $m,n,m+n$ are negative.
  
  If none, then the magnitudes also are $m,n,m+n$, so the claim holds. Hence it also holds if all three are negative.
  
  Now suppose just one is negative. It cannot be $m+n$, so, without loss of generality, we can take it to be $m$.
  Then the three magnitudes are $-m, n, m+n$ and the claim holds. Hence it also holds if two are negative.
  
  Now let us return to the Triangle Inequality for $\Z$, which reduces to
  \[
    |\Mag{(m+n)}| \leq |\Mag {m}| + |\Mag {n}|
    \text{.}
  \]
  If $\Mag{(m+n)}$ is the sum of the other two, then this is an instance of the Triangle Inequality for $\N$, while if $\Mag {m}$ or $\Mag {n}$ is the sum of the other two then it is the Subtractive Inequality. 
\end{proof}
	
	
\end{remark}

Next, let us recall the definition of a geometric field:

\begin{definition}\label{def:geometricfield} A discrete ring $R$ is called a \textit{geometric field} if it satisfies the following:
	\begin{enumerate}[label=(\roman*)]
		\item $0\neq 1$
		\item For any $x\in R$, either $x=0$ or $ \exists x^{-1}\in R$ such that $x\cdot x^{-1}=1$.
	\end{enumerate}
\end{definition}	

\begin{remark}\label{rem:geomfield}
  Classically, Definition~\ref{def:geometricfield} is a plain statement of what it means to be a field.
  Geometrically, the requirement for $R$ to be discrete is important, as the definition does not work for a topological field such as $\R$.
  In fact, it is easy to show that a geometric field must have decidable equality~\cite[Lemma 2.1]{JohnstoneSpectra}.
\end{remark}	

\begin{proposition} \label{prop:fieldav}
	Let $R$ be a geometric field.
	Any absolute value $\gav$ on $R$ is Dedekind and positive definite.
	
	In fact, $[\Alav{R}] \cong [\ADavpos{R}]$. 
\end{proposition}
\begin{proof}
	Suppose $x\in R$.
    If $q$ is rational, then we can define $q<|x|$ when $q<0$, or there are $y\in R$ with $xy=1$, \emph{and} $r'\in\Q$ with $|y|<r'$ (so $0<r'$) and $q<r^{\prime-1}$.
    Clearly this defines a lower real; we must show that, together with the given relation $|x|<r$, it defines a Dedekind.
    
    For disjointness, suppose $q<|x|$ (with $y$ and $r$ as above), and $|x|<q$. Then $1=|x||y| < qr'<1$, a contradiction.
    
    For locatedness, suppose $a<b$ are rationals.
    If $a<0$ then $a<|x|$ by definition.
    If $a=0$, then replace $a$ by $b/2$. Hence we can assume $0<a$.
    Now if $x=0$ then $|x|=0 < a<b$,
    while if $x$ has an inverse $y$, then $|x||y| = 1 < b/a$, so we can find $r,r'$ with $|x|<r$, $|y|<r'$,  and $rr'= b/a$.
    If $r\leq b$ then $|x|<b$ and we are done. If $b<r$ then $rr' < r/a$, so $a < 1/r'$ and $a<|x|$.
    
	
	Positive definiteness follows from the fact that if $xy=1$ then $|x| = |y|^{-1} > 0$.
	
	We have now shown that, from an absolute value, we can define the unique structure needed to make it Dedekind and positive definite. Roughly speaking the extra opens $q < |x|$ are provided using $|x^{-1}|<q^{-1}$.
	This gives the map $[\Alav{R}] \to [\ADavpos{R}]$, and it is clearly inverse to the forgetful map in the opposite direction.
\end{proof}

\subsection{Archimedean vs. Non-Archimedean Absolute Values on $\Z$}

\begin{definition}\label{def:upperAvsNA}
  Let $\gav$ be an absolute value on $\Z$.
  We shall write $\Znq$ for the set of non-negative integers.
  \begin{enumerate}[label=(\roman*)]
  \item
    $\gav$ is \emph{non-Archimedean (NA)} if $|n|<1$ for some $n\in \Znq$.
    
    Immediately we see that the non-Archimedeans form an open subspace of $[\lav]$ -- it's the join of the open subspaces $|n| <1$.
    We write $[\lav_{NA}]$ for it.
  \item
    $\gav$ is \emph{Archimedean (A)} if $1\leq |n|$ for every $n\in \Znq$.
    
    This is a meet of closed subspaces for $1 \leq |n|$, and in fact is the closed complement of $[\lav_{NA}]$.
    We write $[\lav_A]$ for it.
  \item
    $\gav$ is \emph{ultrametric} if
    $|m+n|\leq \max\{|m|, |n|\}$ for every $m,n\in \Znq$.
    
    The ultrametrics do form a subspace of $[\lav]$, which we shall write $[\lav_U]$, but it is neither open nor closed. It is the meet of an infinite family of subspaces described by the sequent
    \[
      \forall m,n\in\Z, q\in\Q(|m|<q \wedge |n| <q
        \rightarrow |m+n| < q)
      \text{.}
    \]
  \end{enumerate}
\end{definition}

\emph{We have departed from the usual convention here.} Some guiding observations:
\begin{itemize}
	\item Normally $\gav$ would be considered Archimedean only if it has some $|n| > 1$, but we cannot express that for $|n|$ as upper real.
	For us, the trivial absolute value has been reassigned as Archimedean, though this does not make a huge difference to the classical account, insofar as that frequently excludes the trivial absolute value as a special case.
	\item 
	Notice: the trivial absolute value is still ultrametric, so ultrametric and non-Archimedean are no longer synonymous.
	(In Lemma~\ref{lem:NAisUltra}, however, we shall prove that every non-Archimedean is ultrametric, as expected.)
	\item 
	The fact that $[\lav_{NA}]$ and $[\lav_{A}]$ are now open and closed complements makes the mathematics significantly smoother for us.
	In particular, it enables us to use (via Lemma~\ref{lem:casesplitting}) case-splitting techniques that, more generally, are invalid geometrically.
	
\end{itemize}

\section{Preliminary lemmas for absolute values on \texorpdfstring{$\Z$}{Z}}
\label{sec:LemmasAV}
In this Section we gather together some preliminary lemmas that are explicit or implicit in the treatment of~\cite{vdW2}.

The first is adjusted for upper reals, lacking the ability to divide.

\begin{lemma}\label{lem:ostrowineq}
  Let $\alpha,\beta$ be positive Dedekinds, and $\gamma,\gamma'$ be non-negative upper reals such that
  \[
    \gamma^v\leq (\alpha v + \beta)\cdot (\gamma')^{v}
  \]
  for all $v\in\mathbb{N}_+$.
  Then $\gamma\leq \gamma'$.
\end{lemma}
\begin{proof}
  First, a basic but key observation: if $\gamma,\gamma'\in\overleftarrow{[0,\infty]}$  such that $\gamma\leq (1+\delta)\gamma'$ for all positive rationals $\delta$, then this implies $\gamma\leq \gamma'$.
  It thus suffices to prove that the Lemma's hypothesis implies $\gamma\leq(1+\delta)\gamma'$, for all positive rationals $\delta$.

	Fix such a rational $\delta>0$. Binomial expansion yields the inequality $(1+\delta)^v\geq 1+v\delta + \frac{v(v-1)}{2}\delta^2 $ for any integer $v\geq 2$. It is clear that for sufficiently large $v$, we get
	\[v\delta >\beta \qquad \text{and} \qquad \frac{1}{2}(v-1)\delta^{2}>\alpha,\]
	and so
	\begin{equation}\label{eq:ostrowineq}
		\gamma^v\leq (\alpha v +\beta )\cdot (\gamma')^v 
          <(1+\delta)^v\cdot (\gamma')^v
          =\left((1+\delta) \gamma'\right)^v.
	\end{equation}
	By~\cite[Lemma 2.12]{NV}, $\argu^v$ reflects non-strict order on non-negative upper reals:
    in other words, from Equation~\eqref{eq:ostrowineq} we can deduce $\gamma\leq (1+\delta)\gamma'$ as required.
\end{proof}	

The second Lemma encapsulates a formula in the main proof in~\cite{vdW2}.
We have clarified its role by varying the base of the logarithms.

\begin{lemma}[Fundamental Lemma for Ostrowski]
\label{lem:ostrowbound} 
	Let $a,b>1$ be any pair of integers greater than 1, and let $\gav$ be any absolute value on $\Z$. 
	
	Then:
	\begin{enumerate}[label=(\roman*)]
		\item $\log_a |a|\leq \max\{0,\log_{b}|b|\} \leq 1$
		\item $\max\{0,\log_b|b|\}=\max\{0,\log_a|a|\}$ for any $a,b>1$.
	\end{enumerate}
	In particular, we can associate a constant upper real $M_{\gav}:=\max\{0,\log_b|b|\} \leq 1$ to any absolute value $\gav$ since by (ii) we know $M_{\gav}$ is independent of our choice of $b>1$.
\end{lemma}	
\begin{proof}
	(i):
	First, note that for any $n\in\N_+$, the triangle inequality yields
    \begin{equation}\label{eq:triangleineq}
	  |n|=|\underbrace{1+1+\dots +1}_{n}| \leq\underbrace{|1|+|1|+\dots +|1|}_{n}= n.
    \end{equation}
	Hence, if $b > 1$, we have $\log_b |b| \leq 1$.
	
	Now, given any pair of integers $a,b$ such that $a,b>1$, and given any $v\in\mathbb{N}_+$, we may expand $a^v$ in powers of $b$ as follows:
	\begin{equation}\label{eq:decomposition}
		a^v= c_0 + c_1b + ...  + c_{r}b^{r}
	\end{equation}
	where $0\leq c_i <b$ for $0\leq i\leq r$, and $c_r\neq 0$. It is obvious that
	\[b^r\leq a^v,\]
	which (taking $\log_b\argu$ on both sides) yields
	\begin{equation}\label{eq:nbound}
		r\leq \log_b a^v = v\log_ba.
	\end{equation}
	Hence, Equation~\eqref{eq:decomposition} and the triangle inequality give
	\begin{align*}\label{eq:bound}
		|a^v|&\leq |c_0| + |c_1||b| + ... + |c_n||b|^{r}\\
		&\leq b(1+|b|+\dots + |b|^r) \leq b(r+1) B^r
	\end{align*}
	where $B=\max\{1,|b|\}$. By Equation~\eqref{eq:nbound}, we get
	\[
      |a|^v=|a^v|
        < b (v\log_ba + 1)\cdot (B^{\log_b a})^{v}
      \text{.}
    \]
	Applying Lemma~\ref{lem:ostrowineq}, this yields
	\[|a|\leq B^{\log_ba},\]
	which in turn yields
	\begin{align*} 
		\log_a |a|&\leq \max\{0,\log_a |b|^{\log_b a}\}\\
		&=\max\{0, \log_b a^{\log_a|b|}\}=\max\{0,\log_b|b|\}.
	\end{align*}
	
	To prove (ii), note that (i) yields for any pair of integers $a,b>1$
	\[\max\{0,\log_a |a|\}\leq \max\{0,\max\{0,\log_b|b|\}\}=\max\{0, \log_b|b| \},\]
	and so by symmetry 
	\[\max\{0,\log_a|a|\}=\max\{0,\log_b|b|\}.\]
\end{proof}

Our third Lemma is that non-Archimedean absolute values have the ultrametric property.
It is important to keep in mind our Definition~\ref{def:upperAvsNA}.
The two terms are no longer synonymous, because the trivial valuation is ultrametric without being non-Archimedean.

\begin{lemma} \label{lem:NAisUltra}
  If an absolute value $\gav$ on $\Z$ is non-Archimedean then it is ultrametric.
\end{lemma}
\begin{proof}
  For some $n$ we have $|n|<1$, so $\max\{0, \log_n |n|\}=0$.
  Then by Lemma~\ref{lem:ostrowbound}, $|m|\leq 1$ for all $m$.
  We can now apply Lemma~\ref{lem:ostrowineq} to show
  $|a+b|\leq \max\{|a|,|b|\}$.
  For any $v\in\mathbb{N}_+$, we have $|{v \choose i}|\leq 1$ and so binomial expansion yields
  \[
    |a+b|^v \leq
      |a|^v + |a|^{v-1}|b| + \cdots + |b|^v\leq (v+1)\cdot \left(\max\{|a|,|b|\}\right)^v \text{.}
  \]
\end{proof}

\section{Ostrowski's Theorem for $\Z$}
\label{sec:MainThm}

Ostrowski's Theorem is typically phrased as a classification result --- it answers the question: \textit{`What are all the non-trivial places of $\mathbb{Q}$?'}. Here, we sharpen this to a representation result for absolute values on $\Z$:

\begin{theorem}[Ostrowski's Theorem for {{$\Z$}}]\label{thm:ostrowskiN} As our setup, denote:
	\begin{itemize}
		\item $[\lav]:=$ The space of absolute values on $\Z$, valued in upper reals.
		\item $\ISpec(\Z):=$  The space of prime ideals of $\Z$.
		\item $\Znq:=$ The set of non-zero integers.
		\item $\overleftarrow{[-\infty,1]}:=$ The space of upper reals bounded above by $1$.
	\end{itemize}
	Define 
	\[\FL:=\{(\frap,\lambda)\in  \ISpec(\Z)\times\overleftarrow{[-\infty,1]} \,\,\big| \,\,\lambda < 0 \leftrightarrow \exists a\in\Znq. (a\in\mathfrak{p})\}.\]	
	Then the following spaces are equivalent:
	\[[\lav]\cong \FL.\]
\end{theorem}

Informally, Theorem~\ref{thm:ostrowskiN} says: any absolute value $\gav$ of $\Z$  can be canonically associated to a pair
\[(\frap,\lambda)\in \ISpec(\Z)\times \overleftarrow{[-\infty,1]}\]
satisfying certain compatibility conditions. 

Before we begin the proof, some preparatory remarks. One, the standard proofs of Ostrowski's Theorem (spelled out in, e.g. \cite{vdW2}) can be adapted to our setting, but they still only give us one direction of the isomorphism. Additional work is therefore needed to construct the second direction, and to show that the two directions are inverse to each other. Two, the decision to work with upper reals (as opposed to Dedekinds) makes the algebra rather delicate, for reasons already alluded to in Observation~\ref{obs:absvalueQ}. Some care is needed in order to maintain geometricity throughout the proof. For the expert reader, let us remark that this (interestingly) results in a picture of $[\lav]$ that is similar but slightly different from the classical picture of the Berkovich spectrum $\M(\Z)$.\footnote{Recall: the Berkovich Spectrum on $\Z$ is defined as the space of multiplicative seminorms on $\Z$, but is equipped with the Gel'fand topology, unlike the non-Hausdorff topology that emerges here due to our use of upper reals.}

\subsection{First Direction: Classification of Absolute Values} \label{subsubsec:1stdirection}
First we define a map
\begin{align*}
	\widehat{f}\colon [\lav]&\longrightarrow \FL\\
	\gav&\longmapsto (\frap_{\gav},\lambda_{\gav})
\end{align*}
by
\begin{itemize}
	\item $\frap_{\gav}:=\{n\in\Z \big| |n|<1\}$
	\item $\lambda_{\gav}:=\inf\{\log_m |m|\big|\,\,  m\in\PSpec \}$.
\end{itemize}
Why is $\lambda_{\gav}$ a geometric construction? Note that each $m$ is a Dedekind real, so by~\cite{NV} we have a map
$\log_m\colon \overleftarrow{[0,\infty)} \rightarrow \overleftarrow{[-\infty,\infty)}$.
Then $\lambda_{\gav}$ is defined as an infimum of upper reals, and so (see Remark~\ref{rem:infsupDed}) itself is also an upper real.

To show $\frap_{\gav}$ is a prime ideal, the only point of difficulty is to show closure under addition and multiplication.
\begin{itemize}
	\item \emph{Closure under Addition.} Suppose $m,n\in \frap_{\gav}$. If either is zero, then immediately $m+n\in\frap_{\gav}$.
	If both are non-zero, then by definition $\gav$ is non-Archimedean and hence (Lemma~\ref{lem:NAisUltra}) ultrametric, so $|m+n| \leq \max\{|m|, |n|\} < 1$.
	\item \emph{Closure under Multiplication.} Suppose $m\in \frap_{\gav}$, and $n\in\Z$. If either is zero, then $m\cdot n \in\frap_{\gav}$. If $m\neq 0$, then $\gav$ is non-Archimedean. Assume that $m>1$. If $n=1$, it is obvious $m\cdot n \in\frap_{\gav}$. If $n>1$, the Fundamental Lemma~\ref{lem:ostrowbound} yields
	$$\log_n|n|\leq \max \{0,\log_m|m|\}=0,$$
	and so $|n|\leq 1$. As such, $|m\cdot n|=|m||n|<1$ and so $m\cdot n \in\frap_{|\cdot|}$. Finally, if instead $m<-1$ and/or $n<-1$, the same argument still holds since $|m|=|-m|$ and $|n|=|-n|$. 
\end{itemize}

Now we know that $\frap_{\gav}$ is prime, by Lemma~\ref{lem:primesemiringideal} it has a non-zero element iff there is a (unique) non-zero prime $p$ for which $|p|<1$, i.e. $\log_p |p|<0$, and this holds iff $\lambda_{\gav}<0$.
Hence $(\frap_{\gav}, \lambda_{\gav})$ is in $\FL$
and we have completed the construction of $\widehat{f}$.

In the Archimedean and non-Archimedean cases, we can recover $\gav$ from $\widehat{f}(\gav)$.

\begin{proposition} \label{prop:avfromfhat}
  Let $\gav$ be an absolute value on $\Z$.
  \begin{enumerate}[label=(\roman*)]
  \item
    $\gav$ is non-Archimedean iff $\lambda_{\gav}<0$.
    In that case, $\frap_{\gav}=(p)$ for some unique prime $p$,
    $\lambda_{\gav}=\log_p|p|$,
    and for all $n\in \N_+$ we have
	\[
	  |n| = \left( p^{\ord_p(n)}\right)^{\lambda_{\gav}}
	  \text{.}
	\]
  \item
    $\gav$ is Archimedean iff $0 \leq \lambda_{\gav}$.
    In that case, $\frap_{\gav}=(0)$, and for all $n\in \N_+$ we have
    \[
      |n| = n^{\lambda_{\gav}}
      \text{.}
    \]
  \end{enumerate}
\end{proposition}
\begin{proof}
  In each case, the facts about $\frap_{\gav}$ and the sign of $\lambda_{\gav}$ have already been noted.
  
  (i): Suppose $\gav$ is non-Archimedean.
  By multiplicativity,
  \[
    |n| = \prod_{q \text{ prime}} |q|^{\ord_q(n)}
    \text{.}
  \]
  If $q$ is a prime not equal to $p$, then $q\notin \frap_{\gav}$ and so $1 \leq |q|$.
  On the other hand, $\log_q |q| \leq \max\{0, \log_p |p|\} = 0$ by  Lemma~\ref{lem:ostrowbound}, so $|q| \leq 1$.
  We deduce that $|q|=1$, and it follows that
  $\lambda_{\gav}=\log_p |p|$ and
  \[
    |n| = |p|^{\ord_p(n)}
      = \left( p^{\log_p |p|}\right)^{\ord_p(n)}
      = p^{\lambda_{\gav}\cdot \ord_p(n)}
    \text{.}
  \]
  (ii):
  Suppose $\gav$ is Archimedean. Then for all $n\geq 2$ in $\Z_+$ (the case $n=1$ is trivial) we have $\log_n |n| = \max\{0, \log_n |n|\}$, which is the constant $M_{\gav}$ of Lemma~\ref{lem:ostrowbound}. It follows that $\lambda_{\gav}$ is this constant, so
  $|n| = n^{\log_n |n|} = n^{\lambda_{\gav}}$.
\end{proof}



\begin{discussion}[Topological Constraints by Upper Reals]\label{dis:nonArchUPPER REALS}
  Unlike the classical Ostrowski's Theorem, Proposition~\ref{prop:avfromfhat} does not directly prove that any non-Archimedean absolute value is equivalent to $\gav_p$ for some prime $p\in\mathbb{N}_+$.  How come? Suppose, given some non-Archimedean $\gav$,  we want a real $\lambda_{\gav}$ such that $\gav=\gav^{\lambda_{\gav}}_p$. Since $|p|_p=p^{-1} = \frac{1}{p}$, the relevant exponent would be $\lambda_{\gav}=\log_{\frac{1}{p}}|p|$, which we know to be a positive \textit{lower real} by \cite[Remark 4.4]{NV}. Geometrically, it is more natural to use the signed upper real $\lambda$ to give a uniform treatment. 
\end{discussion}

\begin{discussion}[Restricting to $\N$]\label{dis:FormalSubissues} 
  Let us return to Remark~\ref{rem:toposabsvNATNUM}.
  The suggestion there is that we might be able to deal with absolute values entirely through their restriction to positive integers, so long as we use subtractive versions of the definition of ideals (Remark~\ref{rem:NATNUMprimeideals}) and the triangle inequality (Proposition~\ref{prop:avN}).
  
  By Proposition~\ref{prop:avN} it is indeed enough to define an absolute value by its restriction to $\N$.
  It can immediately be extended to $\Z$, after which we have our proof of Ostrowski.
  However, there still remains the question whether that proof itself can be framed in terms of the absolute values on $\N$.
  
  There is a definite obstruction to making this work, and we do not know whether it can be circumvented.
  In defining $\widehat{f}$, we need to show that $\frap_{\gav}$ is a prime ideal (in the sense of $\N$), and for that we use Lemma~\ref{lem:NAisUltra} to show that if $\frap_{\gav}$ is non-trivial then $\gav$ is ultrametric.
  But to work with $\N$, the ultrametric property needs to have a subtractive version,
  \[
    |m| \leq \max\{|m+n|, |n|\} \text{.}
  \]
  Together with the usual version, they say that the two largest of $|m|$, $|n|$ and $|m+n|$ are equal, which is clear once Theorem~\ref{thm:ostrowskiN} has been proved. (Consider the two smallest of $\ord_p(m)$, $\ord_p(n)$, and $\ord_p(m+n)$.) Before then, however, the proof of Lemma~\ref{lem:NAisUltra} does not readily adapt to the subtractive version.
\end{discussion}

\subsection{Second Direction: $(\frap,\lambda)$ determines an Absolute Value} \label{subsubsec:2nddirection}

We now work to define a map
\begin{align*}
\widehat{g}\colon \FL &\longrightarrow [\lav]\\
(\frap,\lambda)&\longmapsto \gav_{\frap,\lambda}
\end{align*}
inverse to $\widehat{f}$.

In fact, Proposition~\ref{prop:avfromfhat} already tells us how to do this in the complementary Archimedean and non-Archimedean cases. However, the fact that we require this construction to be geometric creates subtleties. In particular, it is not decidable if $\frap$ is non-trivial or $(0)$, nor is it decidable if $\lambda\leq \lambda'$ for any $\lambda,\lambda'\in\overleftarrow{[-\infty,1]}$. Hence, given $(\frap, \lambda)\in (\Frap,\Lambda)$, our desired map $\widehat{g}$ cannot be directly defined using the following case-splittings:
\begin{itemize}
	\item \textbf{Case 1:} $\lambda < 0$, \textbf{Case 2:} $0\leq \lambda$; or
	\item \textbf{Case 1:} $\exists a\in\Znq$ such that $ a\in\frap$, \textbf{Case 2:} $\frap=(0)$.
\end{itemize}
In other words, the natural faultlines along which one might split $(\frap,\lambda)\in\FL$ into the Archimedean vs. non-Archimedean case does not work in a naive way (i.e. by using decidability to make the split as a coproduct before defining the map separately on the two components). 

As such we must first find a geometric construction that covers both cases in a uniform way.

\begin{construction}\label{cons:GETABSVALUE}
	Suppose $(\frap,\lambda)\in\FL$.
    We define $\gav_{\frap,\lambda}\colon\Z \rightarrow \overleftarrow{[0,\infty)}$ as
	\[
	  |n|_{\frap,\lambda}=
	  \begin{cases}
	    0 & \text{if $n=0$} \\
	    \min\left\{
	      1,
	      \inf \{ \left(p^{\ord_p(n)} \right)^\lambda \mid p \text{ prime in } \frap\}
	    \right\}
	    \boldsymbol{\cdot}
	    \max \left\{ 1, n^\lambda \right\}
	      & \text{if $n>0$} \\
	    |-n|_{\frap,\lambda} & \text{if $n<0$}
      \end{cases}
    \]
\end{construction}

\begin{discussion}\label{dis:ABSVALUEgeometric} Although Construction~\ref{cons:GETABSVALUE} consists of many different components, each component is geometric, and so the final construction is also geometric. To see this, note:
	\begin{itemize}
		\item The initial case split into $n=0$ vs. $n>0$ vs. $n<0$ is permitted since $<$ is decidable on $\Z$.
		\item The $p$-adic ordinal $\ord_p(n)$ is also geometric. This essentially follows from the Euclidean Algorithm, which gives a constructive account of unique prime factorisation for any $n\in\Znq$.
		\item The construction
		\[\inf\{(p^{\ord_p(n)})^\lambda \,\big|\, p \, \text{prime in} \, \frap\}\]
		is a geometric workaround the fact it is undecidable if $\frap=(0)$. Its geometricity comes from
		the fact that upper reals possess arbitrary $\inf$s (Remark~\ref{rem:infsupDed}).
        We shall see how this works in Proposition~\ref{prop:ghatcases}.
		\item Finally, recall from Section~\ref{sec:UpperReals} that $\min$, $\max$ and multiplication are all well-defined operations on the non-negative upper reals. 
	\end{itemize}
\end{discussion}

We can now see how, in the two particular cases corresponding to the non-Archimedean and Archimedean cases in Proposition~\ref{prop:avfromfhat},
the definition splits in the expected way.

\begin{proposition} \label{prop:ghatcases}
  Let $(\frap,\lambda)$ be in $\FL$.
  \begin{enumerate}[label=(\roman*)]
  \item
    If $\lambda<0$ then there is a unique prime $p$ in $\frap$ and, for $n\ne 0$,
    \[
      |n|_{\frap,\lambda} = \left( p^{\ord_p(n)}\right)^\lambda
      \text{.}
    \]
  \item
    If $0\leq \lambda$ then, for $n>0$,
    \[
      |n|_{\frap,\lambda} = n^\lambda
      \text{.}
    \]
  \end{enumerate}
\end{proposition}
\begin{proof}
  (i):
  By the $\FL$ condition and Lemma~\ref{lem:primesemiringideal},
  there is a unique prime $p$ in $\frap$ and,
  for $n>0$, we have both $\left( p^{\ord_p(n)}\right)^\lambda$ and $n^\lambda$ less than 1. It follows that $|n|_{\frap,\lambda}$ evaluates as stated. For $n<0$ the same formula works.

  (ii):
  If $0 \leq \lambda$, then $\frap$ must be $(0)$.
  Then the $\inf$ in the expression is empty, which evaluates as $\infty$.
  We have $1 \leq n^\lambda$, so the entire expression evaluates as stated.
\end{proof}

With this observation we can use the Case-Splitting Lemma~\ref{lem:casesplitting} to show that $\widehat{g}$ maps to $[\lav]$.

\begin{proposition} \label{prop:ghattoav}
  If $(\frap,\lambda)$ is in $\FL$, then $\gav_{\frap,\lambda}$ is an absolute value, and moreover
  $\widehat{f}(\gav_{\frap,\lambda}) = (\frap,\lambda)$.
\end{proposition}
\begin{proof}
  We know that $\gav_{\frap,\lambda}$ is a map from $\Z$ to $\overleftarrow{[-\infty,\infty)}$, and $[\lav]$ is a subspace of the space of such maps. We are trying to show that the inverse image of that subspace is the whole of $\FL$, and for that it suffices to show that it contains both the open subspace where $\lambda<0$ and its closed complement, where $0\leq \lambda$. This is what Lemma~\ref{lem:casesplitting} tells us.
  
  In the light of Proposition~\ref{prop:ghatcases}, the multiplicative property is clear for both cases. It is the triangle inequality that we need to address.
  
  \textbf{Case 1,} $\lambda<0$:
  In fact we expect, and shall prove, the ultrametric inequality.
  Suppose $m,n\neq 0$ (it is trivial if either is zero), with $m=z_1p^{r_1}$ and $n=z_2p^{r_2}$, where $z_1,z_2\in\Znq$, $r_1,r_2\in\N$ and $\gcd(z_1,p)=1=\gcd(z_2,p)$.
  We might as well suppose $r_1\leq r_2$, so
  \[
    m+n= (p^{r_1})\cdot (z_1+z_2p^{r_2-r_1})\text{,}
  \]
  and $\ord_p(m+n) \geq r_1$.
  It follows (using $\lambda < 0$) that
  \[
    |m+n|_{\frap,\lambda} \leq p^{r_1 \lambda}
      \leq \max(|m|_{\frap,\lambda}, |n|_{\frap,\lambda})
    \text{.}
  \]
  
  \textbf{Case 2,} $0 \leq\lambda$:
  By Proposition~\ref{prop:ghatcases}, we need to show that
  \begin{equation}\label{eq:triangleLAMBDA}
	(m+n)^\lambda\leq m^\lambda + n^\lambda.
  \end{equation}

  Suppose we have positive rationals $0<q,t\leq 1$. Since $t^\argu$ is antitonic with respect to rational exponents, this implies that $t \leq t^q$, and so
  \[
    m^qt\leq m^qt^q, \qquad \text{for any positive integer $m>0$.}
  \]
  This in turn implies
  \begin{equation}\label{eq:triQ}
    \begin{split}
	  (m+n)^q &= (m+n)^q \left(\left(\frac{m}{m+n}\right) + \left(\frac{n}{m+n}\right)\right) \\
	   & \leq (m+n)^q \left(\frac{m}{m+n}\right)^q +(m+n)^q \left(\frac{n}{m+n}\right)^q = m^q+n^q
	   \text{.}
    \end{split}
  \end{equation}
  The calculation here is intrinsically two-sided, relying on the assumption that $q$ is rational: for an upper real $\lambda$ instead of $q$, we should be trying to multiply, for example, an upper real $(m+n)^\lambda$ by a lower real
  $\left(\frac{m}{m+n}\right)^\lambda$,
  and that is not well-defined (cf. the remark on monotonicity in Section~\ref{sec:UpperReals}).

  However, the \emph{conclusion} $(m+n)^q \leq m^q + n^q$ still makes sense for upper reals, with $\leq$ now the (opposite of the) specialisation order.
  Hence for an upper real $\lambda$ we have
  \[
    (m+n)^\lambda = \inf_{\lambda<q} (m+n)^q
      \leq \inf_{\lambda<q} (m^q+n^q)
      = \inf_{\lambda<q} m^q + \inf_{\lambda<q} n^q
      = m^\lambda + n^\lambda
    \text{.}
  \]


To summarise: $|\cdot|_{p,\lambda}$ is an absolute value in both Cases 1 and 2. By Lemma~\ref{lem:casesplitting}, this defines a map $\hat{g}\colon \FL\to[\lav]$. For the rest of the result we look at the equaliser of $\widehat{f}\circ\widehat{g}$ and the identity,
  a subspace of $\FL$. Again, to show it is the whole of $\FL$ it suffices to show the equation is satisfied in both the two cases.
  
  \textbf{Case 1,} $\lambda<0$:
  For $\frap$, we have
  \[
    \left( p^{\ord_p(n)}\right)^\lambda < 1
      \Leftrightarrow p^{\ord_p(n)} > 1
      \Leftrightarrow \ord_p(n) >0
      \Leftrightarrow p \text{ divides } n
      \Leftrightarrow n\in\frap
      \text{.}
  \]

  For $\lambda$, consider that for a prime $q$ we have
  \[
    \log_q \left( \left( p^{\ord_p(q)} \right)^\lambda \right) =
      \begin{cases}
	    \log_p \left(p^\lambda\right) = \lambda
	    & \text{if $q=p$,} \\
	    \log_q 1 = 0 & \text{if $q\ne p$.}
      \end{cases}
  \]
  Hence the $\inf$ of these is $\lambda$.
  
  \textbf{Case 2,} $0 \leq\lambda$:
  For $\frap$, we have $n^\lambda < 1$ iff $n=0$.
  
  For $\lambda$,
  \[
    \log_q\left( q^\lambda \right) = \lambda
    \text{,}
  \]
  and the $\inf$ of these is $\lambda$.
\end{proof}

It remains to show that $\widehat{g}\circ \widehat{f}$ is the identity.
Again it suffices to verify this on the non-Archimedeans and the Archimedeans, and that was done in Propositions~\ref{prop:avfromfhat} and~\ref{prop:ghatcases}.

\textbf{This concludes the proof of Theorem~\ref{thm:ostrowskiN}.}\newline 

Notice: the open-closed gluing of $[\lav_{NA}]$ to $[\lav_A]$ in $[\lav]$ transfers to $\FL$.

\begin{corollary} \label{cor:mainGluing}
\[
  \begin{split}
    [\lav_{NA}] &\cong \PSpec \times \overleftarrow{[-\infty, 0)} \\
    [\lav_A] &\cong \overleftarrow{[0, 1]}
  \end{split}
\]
\end{corollary}
\begin{proof}
  The main Theorem splits up using Propositions~\ref{prop:avfromfhat} and~\ref{prop:ghatcases}.
\end{proof}

	\begin{discussion}[Undecidability and Geometricity]\label{dis:casesplittingGEOMETRICPROOF} Given our previous remarks about decidability issues, our proof of Theorem~\ref{thm:ostrowskiN} highlights an interesting subtlety. Namely, why is undecidability a barrier to geometricity when defining constructions, but not when proving properties? Examining the hypotheses of the Case-Splitting Lemma~\ref{lem:casesplitting}, a key tool in our proof, reveals an interesting fine print. The Lemma only justifies analysing a construction via case-splitting \textit{once} the construction already exists (i.e. the map $X\xrightarrow{f}Z$); it does not justify defining \textit{a new construction} via case-splitting, as one might have hoped.  
\end{discussion}

\section{Absolute Values on $\Q$}
\label{sec:Qav}


We now turn to absolute values on $\Q$, with a view to recovering Ostrowski's theorem in its usual form.
Straight away, Proposition~\ref{prop:fieldav} tells us that $[\Alav{\Q}] \cong [\ADavpos{\Q}]$,
so, as far as $\Q$ is concerned, no extra generality is obtained by using the upper reals.

However, we can still exploit Theorem~\ref{thm:ostrowskiN} in a non-trivial way by observing that an absolute value on $\Q$ can be derived from its restriction to $\Z$ -- which, again, is Dedekind and positive definite.

\begin{proposition} \label{prop:QZDpos}
  The restriction map $[\ADavpos{\Q}] \to [\ADavpos{\Z}]$ is an isomorphism.
\end{proposition}
\begin{proof}
  For the inverse, suppose $\gav$ is a positive definite, Dedekind absolute value on $\Z$.
  If $n\ne 0$ then $|n|$ is positive, and hence invertible. Thus, to extend $\gav$ to $\Q$ we must define $|m/n| = |m|/|n|$. This is easily checked to be well defined and a positive definite, Dedekind absolute value on $\Q$, restricting to the original $\gav$ on $\Z$.  
\end{proof}

We then have the forgetful map to $[\Alav{\Z}]$, and our aim is to exploit its isomorphism with $\FL$.
Clearly, we shall be interested in allowing the $\lambda$ components in $\FL$ to be Dedekind.

\begin{definition}[cf. Theorem~\ref{thm:ostrowskiN}] \label{def:FLDed}
   
   	Define 
   \[
     \FLDed :=\{ (\frap,\lambda) \in
       \ISpec(\Z)\times (-\infty,1] \mid
       \lambda < 0 \leftrightarrow \exists a\in\Znq. (a\in\mathfrak{p}) \}
       \text{.}
   \]
\end{definition}

Thus we have
\begin{equation} \label{eq:PLdia}
  \begin{tikzcd}
    {[\Alav{\Q}]}
      \ar[r, "\cong"]
    & {[\ADavpos{\Q}]}
      \ar[r, "\cong"]
    & {[\ADavpos{\Z}]}
      \ar[r]
    & {[\Alav{\Z}]}
      \ar[d, "\cong"]
    \\
    && \FLDed
      \ar[r]
    & \FL
  \end{tikzcd}
\end{equation}

An obvious conjecture is that the isomorphism of Theorem~\ref{thm:ostrowskiN} lifts to one between $[\ADavpos{\Z}]$ and $\FLDed$.
Encouragingly (Propositions~\ref{prop:QnonArch} and~\ref{prop:QArch}), the isomorphisms of Corollary~\ref{cor:mainGluing} \emph{do} lift,
and in Theorem~\ref{thm:ostrowskiQ} we use this to obtain the standard Ostrowski result for non-trivial Dedekind absolute values.

Unfortunately (Observation~\ref{obs:nolambdaDed}), the isomorphisms do not transfer across the boundary at the trivial absolute value.
The basic problem is that an infimum of Dedekind reals has an upper part, but is not necessarily Dedekind.
Unexpectedly, we find that $[\ADavpos{\Z}]$ and $\FLDed$ are two distinct, non-isomorphic open-closed gluings of $\PSpec\times (-\infty, 0)$ to $[0,1]$.
We do not fully understand what is going on here, nor what it has to say about number theory.

We are already writing $[\lav]$ for $[\Alav{\Z}]$.
Let us also write $[\Dav]$ for the isomorphic spaces
$[\Alav{\Q}]$, $[\ADavpos{\Q}]$, and
$[\ADavpos{\Z}]$. Recalling Definition~\ref{def:upperAvsNA}, $[\Dav]$ inherits three subspaces from $[\lav]$:
the open subspace of non-Archimedeans, its closed complement of Archimedeans, and the space of ultrametrics.
We shall write them as $[\Dav_{NA}]$, $[\Dav_{A}]$, and $[\Dav_{U}]$.
Keep in mind that our definition of Archimedean is non-standard in that it includes the trivial absolute value.

\begin{definition} \label{def:Archnt}
  An absolute value $\gav$ over $\Q$ is
  \emph{non-trivial Archimedean} if there exists some $n\in\Znq$ such that $|n|>1$. They form an open subspace $[\Dav_{A^+}]$ of $[\Dav]$.
\end{definition}

\begin{proposition} \label{prop:Qavopens}
  ~
  \begin{enumerate}[label=(\roman*)]
  \item
    $[\Dav_{NA}]$ and $[\Dav_{A^+}]$ are disjoint.
  \item
    The closed complement of $[\Dav_{NA}]\vee[\Dav_{A^+}]$
    is $\{\gav_0\}$.
    
    (In the light of this, we shall call $\gav$ in $[\Dav]$ \emph{non-trivial} if it is in
    $[\Dav_{NA}]\vee[\Dav_{A^+}]$.)
  \item
    The closed complement of $[\Dav_{A^+}]$ is $[\Dav_{U}]$.
  \end{enumerate}
\end{proposition}
\begin{proof}
  (i)
  follows from Lemma~\ref{lem:ostrowbound}: if we have $|m|<1$, with $m\ne 0$, then we cannot have any $|n|>1$.
  
  (ii)
  For $n\ne 0$, if $|n| < 1$ and $|n| > 1$ are both impossible, then we must have $|n|=1$.
  
  (iii)
  By part (ii), the closed complement of $[\Dav_{A^+}]$ is the subspace join of $[\Dav_{NA}]$ and $\{\gav_0\}$.
  Using Lemma~\ref{lem:NAisUltra}, both these are contained in $[\Dav_{U}]$.
  
Conversely, if $\gav$ is ultrametric and non-trivial Archimedean, with $|n|>1$, then from
  $|n| = |1+\cdots +1| \leq \max(1,\ldots,1) = 1$ we get a contradiction.
  Hence $[\Dav_{U}]$ is contained in the complement of $[\Dav_{A^+}]$. 
\end{proof}

The following two propositions extend our previous work in Section~\ref{sec:MainThm} to the present setting. An important constructive issue is a consequence of Remark~\ref{rem:infsupDed}:
even for a Dedekind absolute value, the upper real $\lambda_{\gav}$ of Section~\ref{subsubsec:1stdirection}, calculated as an $\inf$, is not necessarily Dedekind.
To prove it is, we need to deal with the non-Archimedean and Archimedean cases separately.

\begin{proposition}\label{prop:QnonArch}
  We have an isomorphism
  \[
    [\Dav_{NA}] \cong \PSpec \times (-\infty, 0) \text{.}
  \]
\end{proposition}
\begin{proof}
  If $\gav\in [\Dav_{NA}]$, consider
  $\widehat{f}(\gav) = (\frap_{\gav}, \lambda_{\gav})$.
  Because $\gav$ is non-Archimedean, there is a unique $p\in\PSpec$ such that $\frap_{\gav}=(p)$,
  and this gives a map $[\Dav_{NA}] \to \PSpec$.
  We then have $\lambda_{\gav}= \inf_{q\in\PSpec} \log_q|q| = \log_p|p|$,
  because $p$ is the unique prime for which $|p|<1$,
  and this $\lambda_{\gav}$ is Dedekind in $(-\infty,0)$.
  This gives a map 
  $$F\colon [\Dav_{NA}] \rightarrow \PSpec \times (-\infty,0).$$
  
  In the opposite direction, suppose
  $(\frap, \lambda)\in\FLDed$ has $\frap=(p)$ (and $\lambda<0$).
  Then, by Proposition~\ref{prop:ghatcases}, $\widehat{g}$ maps it to the absolute value with
  $|n|_{(p),\lambda}
    = \left( p^{\ord_p(n)} \right)^\lambda $, which is Dedekind and positive definite. Call the associated map of this construction $G$. Assembling the data, this gives commutative diagrams
    \begin{equation} \label{dia:DavNA}
      \begin{tikzcd}
        {[\Dav_{NA}]}
          \ar[r]
          \ar[d, shift right=1ex, "F"']
        & {[\lav_{NA}]}
          \ar[r, hook]
        & {[\lav]}
          \ar[d, shift right=1ex, "\widehat{f}"']
        \\
        \PSpec \times (-\infty, 0)
          \ar[u, shift right=1ex, "G"']
          \ar[r, hook]
        & \FLDed
          \ar[r]
        & \FL
          \ar[u, shift right=1ex, "\widehat{g}"']
      \end{tikzcd}
    \end{equation}
    Both rows are monic (using Discussion~\ref{dis:Dedtouppermonic}), so
    $F$ and $G$ are inverse to each other because $\widehat{f}$ and $\widehat{g}$ are.
\end{proof}

\begin{corollary} \label{cor:QNonArch}
  For any $\gav\in [\Dav_{NA}]$ with $|p|<1$, $p\in\PSpec$, there exists $\alpha\in(0,\infty)$ such that $\gav=\gav_p^\alpha$.
\end{corollary}
\begin{proof}
  Take $\alpha=-\lambda_{\gav}$ and use Proposition~\ref{prop:ghatcases}.
\end{proof}

\begin{proposition}\label{prop:QArch}
	We have isomorphisms
	\[ [\Dav_{A}]\cong[0,1] \text{,} \]
	\[ [\Dav_{A^+}]\cong (0,1] \text{.} \]
	In particular, for any $\gav\in[\Dav_{A}]$, there exists $\alpha\in[0,1]$ such that $\gav=\gav^\alpha_{\infty}$;
    if $\gav$ is non-trivial then $\alpha > 0$.
\end{proposition}
\begin{proof} The proof is similar to that of Proposition~\ref{prop:QnonArch}. Define the maps
	\begin{align*}
	F\colon [\Dav_A]&\longrightarrow [0,1]\\
	\gav&\longmapsto \log_2|2|,
	\end{align*}	
	and 	
	\begin{align*}
	  G\colon [0,1]&\longrightarrow[\Dav_A]\\
	  \lambda&\longmapsto \gav^\lambda_\infty = \widehat{g}((0),\lambda).
	\end{align*}

  From the Fundamental Lemma~\ref{lem:ostrowbound}, we see that $\log_n|n| = F(\gav)$ for all integers $n > 1$, and so $\widehat{f}(\gav) = ((0), F(\gav))$.
  Also, $G(\lambda)$ is clearly Dedekind and positive definite.
  
  We now have a diagram like that of~\eqref{dia:DavNA}, but with $[\Dav_A]$ and $[0,1]$ on the left. The same argument shows that $F$ and $G$ are mutually inverse. A similar proof shows that $[\Dav_{A^+}]\cong (0,1]$.
\end{proof}

\begin{theorem}[Ostrowski's Theorem for $\mathbb{Q}$]\label{thm:ostrowskiQ} Let $\gav$ be a non-trivial absolute value on $\mathbb{Q}$. Then, one of the following must hold:
	\begin{enumerate}[label=(\roman*)]
		\item $\gav=\gav^\alpha_{\infty}$ for some $\alpha\in(0,1]$; or
		\item $\gav=\gav_p^\alpha$ for some $\alpha\in (0,\infty)$ and some prime $p\in\mathbb{N}_+$. 
	\end{enumerate} 
\end{theorem}

\begin{proof}
  Working with non-trivial absolute values, we are in
  $[\Dav_{NA}]\vee [\Dav_{A^+}]$.
  As it is a join of two disjoint opens, we can validly use case splitting in the geometric reasoning to combine the $F$-maps of Propositions~\ref{prop:QnonArch} and~\ref{prop:QArch} to get a map from $[\Dav_{NA}]\vee [\Dav_{A^+}]$
  to $(0,\infty) \vee (0,1]$.
  Categorically, $[\Dav_{NA}]\vee [\Dav_{A^+}]$ is the coproduct of $[\Dav_{NA}]$ and $[\Dav_{A^+}]$.
  Similarly, we can combine the $G$-maps to get an inverse.
  
  For the second part, we then use Corollary~\ref{cor:QNonArch}.
\end{proof}

\begin{discussion}\label{dis:TopAlgebra} While negation inverts orientation on one-sided reals (and so, a negated upper real becomes a lower real), a negated Dedekind real is still a Dedekind, so Corollary~\ref{cor:QNonArch} avoids the same issues mentioned in  Discussion~\ref{dis:nonArchUPPER REALS}. Further, unlike the one-sided reals, $-\infty,\infty$ are not Dedekinds, which is why $\lambda=\log_{\frac{1}{p}}|p|\in(0,\infty)$, as opposed to $\lambda\in (0,\infty]$. In other words, if we wish to view $\lambda$ as a Dedekind, then we lose the ability to speak about the true seminorms,
the $p$-characteristic absolute values $\pchav{\cdot}{p}$,
which correspond to $\gav^\infty_p$ for prime $p\in\mathbb{N}_+$. This, combined with Observations~\ref{obs:absvalueQ} and \ref{obs:seminormsUPPER}, brings into focus the following connection between the algebra (absolute values) and topology (the reals):
	\begin{itemize}
		\item Multiplicative Seminorms $ \leftrightsquigarrow$ Upper reals
		\item (Positive Definite) Norms $\leftrightsquigarrow$ Dedekind reals.
	\end{itemize}
\end{discussion}

Perhaps surprisingly, the assumption of non-triviality is essential in Theorem~\ref{thm:ostrowskiQ}. As alluded to earlier,
we cannot combine Propositions~\ref{prop:QnonArch} and~\ref{prop:QArch} to get an isomorphism of $[\Dav]$ with $\FLDed$. In fact, there is not even any map from $[\Dav]$ to $\FLDed$ that makes the obvious diagram commute with $\widehat{f}\colon [\lav] \to \FL$.
We prove this by restricting to $[\Dav_{U}]$ and considering the projection from $\FLDed$ to $(-\infty, 0]$.

\begin{observation}\label{obs:nolambdaDed}
	There is no map $F'\colon [\Dav_{U}] \to (-\infty, 0]$ compatible with $\widehat{f}$ restricted to $[\lav_U]$:
    \begin{equation} \label{eq:Fdash}
      \begin{tikzcd}
        {[\Dav_{U}]}
          \ar[r]
          \ar[d, "{F'}"']
        & {[\lav_U]}
          \ar[d, "\widehat{f}"]
        \\
        {(-\infty,0]}
          \ar[r]
        & {\overleftarrow{[-\infty,0]}}
      \end{tikzcd}
    \end{equation}
    Consequently, there is no map from $[\Dav]$ (i.e. $[\Alav{\Z}]$) to $\FLDed$ that makes diagram~\eqref{eq:PLdia} commute.
\end{observation}
\begin{proof}
  The conditions tell us that 
  \[
    F'(\gav) = \inf_{p \text{ prime}} \log_p |p| \text{,}
  \]
  but also that this infimum is to be Dedekind (and not just upper). Let us examine the inverse image
  $F^{\prime-1}(-1, 0]$.
  
  Using multiplicativity and the theory of logarithms, we see that a subbase of opens for $[\Dav_{U}]$ is provided by conditions
  $q < \log_p|p| < r$, where $q$ and $r$ are rationals and $p$ is prime. Since $\gav$ is ultrametric, we can assume $q<0$. Hence a base of opens is provided by finite conjunctions
  \begin{equation} \label{eq:avubasic}
    \bigwedge_{i=0}^n q_i < \log_{p_i} |p_i| < r_i
    \text{.} 
  \end{equation}

  If $F'$ exists, then $F^{\prime-1}(-1, 0]$ is an open containing $\gav_0$ (because $F'(\gav_0)=0$), and so $\gav_0$ is in such a basic.
  We must then have $0<r_i$ for every $i$.
  However, for such a basic it is impossible to be contained in
  $F^{\prime-1}(-1, 0]$.
  For we can define $\gav$ so that $|p|<1$ for some prime $p$ distinct from all the $p_i$s, and $\log_p |p|=-2$. Then that $\gav$ is in our basic, but not in $F^{\prime-1}(-1, 0]$.
  
  Now suppose we have a map from $[\Dav]$ to $\FLDed$ as suggested. We then have the following commutative diagram.
  \[
    \begin{tikzcd}
      {[\Dav_{U}]}
        \ar[r]
        \ar[d,hook]
      & {[\lav_U]}
        \ar[d,hook]
      \\
      {[\Dav]}
        \ar[r]
        \ar[d]
      & {[\lav]}
        \ar[d]
      \\
      \FLDed
        \ar[r]
        \ar[d]
      & \FL
        \ar[d]
      \\
      {(-\infty,1]}
        \ar[r]
      & {\overleftarrow{[-\infty,1]}}
    \end{tikzcd}
  \]
  The vertical composites factor via $(-\infty, 0]$ and $\overleftarrow{[-\infty, 0)}$, and we obtain a map $F'$ as in Diagram~\eqref{eq:Fdash}.
\end{proof}

\begin{discussion}
To test our understanding, consider the following argument:
	\begin{itemize}
		\item $[\Dav_{NA}]$ and $[\Dav_{A}]$ are open/closed complement subspaces in $[\Dav]$.
		\item Propositions~\ref{prop:QnonArch} and~\ref{prop:QArch} characterise $[\Dav_{NA}]$ and $[\Dav_{A}]$.
		\item Hence, by appealing to the Case-Splitting Lemma~\ref{lem:casesplitting} or otherwise, conclude that $[\Dav]\cong\FLDed$.
	\end{itemize}
This proof strategy imitates what we did in Section~\ref{sec:MainThm}, yet Observation~\ref{obs:nolambdaDed} tells us that the argument is wrong. So where does it fail? Recall: in Section~\ref{sec:MainThm}, we defined a map 
\begin{align}
\hat{f}\colon [\lav] &\longrightarrow \FL\\
|\cdot|&\longmapsto (\frap_{|\cdot|},\lambda_{|\cdot|})  \nonumber 
\end{align}
by verifying that $\frap_{|\cdot|},\lambda_{|\cdot|}\in\FL$ \emph{without} any initial assumptions on whether $|\cdot|$ is Archimedean or non-Archimedean. However, our present setting requires $\lambda_{|\cdot|}$ to be a Dedekind, not just an upper real. This was verified in Propositions~\ref{prop:QnonArch} and ~\ref{prop:QArch} but notice this argument involves a case-splitting. This is why the argument in Section~\ref{sec:MainThm} does not extend to yield a map $[\Dav]\to\FLDed$, or indeed $[\Dav_{U}]\to \FLDed$. In addition, the Dedekinds do not form a subspace of the upper reals (cf. Discussion~\ref{dis:Dedtouppermonic}), so the Case-Splitting Lemma cannot be applied to glue  Propositions~\ref{prop:QnonArch} and ~\ref{prop:QArch} together in order to assemble the desired map.
\end{discussion}

A related question looks at the spectra.
The reader may reasonably wonder: does $\ISpec$ still play a role in the analysis of absolute values on $\Q$? After all, $\ISpec$ is not mentioned in the statement of Theorem~\ref{thm:ostrowskiQ} (unlike Theorem~\ref{thm:ostrowskiN}), and only features implicitly in its proof. Might we not e.g. use the Zariski spectrum $\LSpec(\Z)$ instead to denote the non-Archimedean places of $\Q$? 

The answer comes from a closer examination of the maps from $[\Dav]$ to the other spectra. We know already that $[\Dav]$ has a map to $\ISpec$, via $[\lav]$ and $\FL$.
Can we use the Dedekind property to find a map to either $\LSpec$ or $\FSpec$?
Again, we can find a negative answer by restricting to $[\Dav_{U}]$.
Of course, we already know that there is a map from $[\Dav_{NA}]$ to $\PSpec$, and hence to the other three spectra.
The problem comes when we try to bring in the trivial absolute value.



\begin{observation}[coZariski vs. Zariski/Constructible Topology]\label{obs:coZariskiSpectra}
  There is no map from $[\Dav_{U}]$ to $\LSpec(\Z)$
  that maps $\gav_0$ to $\Z-\{0\}$, and each $\gav_p^\alpha$ to the prime filter of integers not divisible by $p$.
  
  It follows that there is no map to $\FSpec(\Z)$ with the corresponding properties.
\end{observation}
\begin{proof}
    
    Suppose there did exist such a map
	\[f\colon [\Dav_{U}]\rightarrow \LSpec(\Z).\]	
	Now consider the open in $\LSpec(\Z)$ comprising all those prime filters containing 2.
    It has the prime filters corresponding to all non-zero primes except 2, as well as the top filter $\Z-\{0\}$.
    Let $U$ be its inverse image under $f$.
    
    Again, $U$ is a join of basics corresponding to propositions as in condition~\eqref{eq:avubasic}.
    Hence $\gav_0$ is in one of those basics, which must have $q_i < 0<r_i$ for every $i$ -- and the values $r_i$ are then irrelevant.
    We can now find a 2-adic $\gav$ such that
    \[
      \log_2 |2| > \max_i q_i \text{,}
    \]
    and this $\gav$ is in our basic open.
    That's a contradiction, since its image under $f$ is a (the) prime filter not containing 2.
\end{proof}

\begin{discussion} Informally,  Observation~\ref{obs:coZariskiSpectra} holds because such a map requires us to decide, given some generic $\gav\in[\Dav_{U}]$, whether   $\gav$ is trivial or non-trivial, which we can only do after verifying that $|p|=1$ for all non-trivial primes $p$. This, however, requires universal quantification and is therefore non-geometric. 
	
\end{discussion}

\section{Further Remarks and Discussion}

Our discussion after Theorem~\ref{thm:ostrowskiQ} gives an overview of how Dedekind absolute values interact with the ``edge cases''. While perhaps unsurprising that $[\Dav]$ does not include the $p$-characteristic absolute values (Discussion~\ref{dis:TopAlgebra}), it was unexpected (at least to the authors) how the trivial absolute value is enmeshed with topological subtleties. Extending Observation~\ref{obs:nolambdaDed}, we pose the obvious test problem:

\begin{problem} Characterise the (entire) space $[\Dav]$, including the trivial absolute value. 
\end{problem}

\begin{discussion} Much has already been said regarding the difficulties reconciling Archimedean vs. the non-Archimedean structures (see e.g. the introduction of \cite{Maz}), which is often viewed as being bound up with the differences between the analytic vs. the algebraic world. Observations~\ref{obs:nolambdaDed} and \ref{obs:coZariskiSpectra} highlight a different perspective: just reconciling the trivial vs. non-trivial non-Archimedean absolute values itself has its own set of difficulties that require some care (at least in the Dedekind setting). 
	
	Perhaps then before we start connecting the Archimedean and the non-Archimedean, one should first figure out how the trivial absolute value interacts with each component. An important clue comes from our Ostrowski's Theorem~\ref{thm:ostrowskiN} for $\Z$: the trivial absolute value acts as a kind of hinge point between the Archimedeans and non-Archimedeans so long as we work with the (honest) upper reals. The challenge then is to adapt this picture to the Dedekind absolute values (Discussion~\ref{dis:Dedtouppermonic} may be relevant here). After which, it will be very interesting to meditate on how our characterisation of $[\Dav]$ may tell us something useful about connecting the analytic with the algebraic.
	
\end{discussion}

\begin{discussion} As alluded to in the introduction, the issue of reconciling trivial vs. non-trivial norms also persists in Berkovich geometry -- see e.g. Example 1.4 in \cite{BerkovichMonograph} or the case-split in \S 3.4 and 3.5 regarding GAGA-type results for trivial vs. non-trivial valuations. A point-free perspective on some of these issues was investigated in \cite[Ch. 5]{MingPhD} -- again, the upper reals play a helpful role.
\end{discussion}

Might the results of the present paper be understood as suggesting that we should abandon Dedekind absolute values and simply work with upper-valued absolute values on $\Z$? We argue no. One important reason is that part of what makes understanding the Dedekind absolute values so interesting is their connection with Local-Global Principles from number theory, which is less clear for the non-Dedekind case.

Let us elaborate. A polynomial $f$ with $\Q$-coefficients is said to satisfy the \emph{Local-Global Principle} just in case that $f$ has rational solutions iff $f$ has solutions over all non-trivial completions of $\Q$ (up to equivalence). There are well-known examples of polynomials that satisfy the Local-Global principle (e.g. quadratic forms) as well as those that fail it (e.g. $3x^3+4y^3+5z^3=0$), and improving our understanding of the obstructions to this principle is a deep and active area of research. 

Now, we say that two Dedekind absolute values $|\cdot|_1,|\cdot|_2$ of $\Q$ belong to the same \emph{place} iff there exists $\alpha\in (0,1]$ such that $|\cdot|_1=|\cdot|^\alpha_2$ or  $|\cdot|_2=|\cdot|^\alpha_1$. In particular, one can verify that any two absolute values from the same place define topologically equivalent completions of $\Q$. Notice then that our Ostrowski's Theorem~\ref{thm:ostrowskiQ} for $\Q$ classifies all non-trivial completions of $\Q$ up to equivalence, as is relevant for the Local-Global Principle. Notice also (e.g. by Discussion~\ref{dis:nonArchUPPER REALS}) that the language of places does not adapt well to the setting of upper-valued absolute values on $\Z$. 

This sets up the following test problem, which will be the focus of a subsequent paper by the authors.

\begin{problem} Characterise $[\mathrm{places}]$, the space of places of $\Q$.
\end{problem}

\begin{discussion}\label{dis:workwithISpec} Observation~\ref{obs:coZariskiSpectra} gives us our first important clue. It is natural to expect there to exist a quotient map
	\[\mathrm{quot}\colon [\Dav_{U}]\longrightarrow [\mathrm{places}_{U}]\]
	that sends an ultrametric absolute value to its corresponding place. However, if $[\mathrm{places}_{U}]=\LSpec(\Z)$ or $\FSpec(\Z)$, then
	Observation~\ref{obs:coZariskiSpectra} tells us that no such map exists. This gives another compelling reason for working with $\ISpec(\Z)$, as opposed to the other spectral spaces.
\end{discussion}

	
	\bibliography{ToposAbsValue}

\begin{thebibliography}{vdW91b}

\bibitem[Ber90]{BerkovichMonograph}
Vladimir Berkovich.
\newblock {\em Spectral Theory and Analytic Geometry over Non-Archimedean
  Fields}.
\newblock American Mathematical Society, 1990.

\bibitem[Col16]{ColeSpectra}
Julian Cole.
\newblock The bicategory of topoi and spectra.
\newblock {\em Reprints in Theory and Applications of Categories}, 25:1--16,
  2016.

\bibitem[Joh77a]{JohnstoneSpectra}
P.~T. Johnstone.
\newblock Rings, fields, and spectra.
\newblock {\em Journal of Algebra}, 49:238--260, 1977.

\bibitem[Joh77b]{J0}
P.T. Johnstone.
\newblock {\em Topos Theory}.
\newblock Academic Press, 1977.

\bibitem[Joh82]{StoneSp}
P.T. Johnstone.
\newblock {\em Stone Spaces}.
\newblock Cambridge University Press, 1982.

\bibitem[Joh02a]{J1}
P.T. Johnstone.
\newblock {\em Sketches of an Elephant: A Topos Theory Compendium}, volume~1.
\newblock Clarendon Press, 2002.

\bibitem[Joh02b]{J2}
P.T. Johnstone.
\newblock {\em Sketches of an Elephant: A Topos Theory Compendium}, volume~2.
\newblock Clarendon Press, 2002.

\bibitem[Maz93]{Maz}
Barry Mazur.
\newblock On the passage from local to global in number theory.
\newblock {\em Bulletin of the AMS}, 29(1):14--50, 1993.

\bibitem[Ng23]{MingPhD}
Ming Ng.
\newblock {\em Adelic Geometry via Topos Theory}.
\newblock PhD thesis, University of Birmingham, 2023.

\bibitem[NV22]{NV}
Ming Ng and Steven Vickers.
\newblock Point-free construction of real exponentiation.
\newblock {\em Logical Methods in Computer Science}, 2022.

\bibitem[vdW91a]{vdW1}
B.L. van~der Waerden.
\newblock {\em Algebra}, volume~1.
\newblock Springer-Verlag, 1991.

\bibitem[vdW91b]{vdW2}
B.L. van~der Waerden.
\newblock {\em Algebra}, volume~2.
\newblock Springer-Verlag, 1991.

\bibitem[Vic05]{ViLoccompI}
Steven Vickers.
\newblock Localic completion of generalized metric spaces {I}.
\newblock {\em Theory and Applications of Categories}, 14:328--356, 2005.

\bibitem[Vic07a]{Vi4}
Steven Vickers.
\newblock Locales and toposes as spaces.
\newblock In M~Aiello, I~E Pratt-Hartmann, and J~F van Benthem, editors, {\em
  Handbook of Spatial Logics}, pages 429--496. Springer, 2007.

\bibitem[Vic07b]{ViSublocales}
Steven Vickers.
\newblock Sublocales in formal topology.
\newblock {\em Journal of Symbolic Logic}, 72(2):463--482, 2007.

\bibitem[Vic22]{VickersPtfreePtwise}
Steven Vickers.
\newblock Generalized point-free spaces, pointwise.
\newblock arXiv:2206.01113, 2022.

\end{thebibliography}

\end{document}